\documentclass[12 pt]{article}
\usepackage{stmaryrd}
\usepackage{times}
\usepackage{booktabs}
\usepackage{subfigure}
\usepackage{adjustbox}
\usepackage{rotating}
\usepackage{diagbox}
\usepackage{tabularx}
\usepackage{multirow, makecell}
\usepackage{rotating}
\usepackage{longtable}
\usepackage{supertabular}
\usepackage{pifont}
\usepackage{floatrow}
\usepackage{caption}
\usepackage{mathrsfs}
\usepackage[fleqn]{amsmath}
\usepackage{amsfonts,amsthm,amssymb,mathrsfs,bbding}
\usepackage{txfonts}
\usepackage{graphics,multicol}
\usepackage{graphicx}
\usepackage{makecell}
\usepackage{color}
\usepackage{caption}
\usepackage{ textcomp }
\usepackage{ stmaryrd }
\usepackage{indentfirst}
\usepackage{cite}
\usepackage{latexsym,bm}
\usepackage{enumerate}
\evensidemargin=\oddsidemargin\topmargin=-1.5cm

\newtheorem{lem}{Lemma}[section]

\newtheorem{thm}{Theorem}[section]

\newtheorem{pro}{Proposition}[section]

\theoremstyle{definition}

\newtheorem{remark}{Remark}[section]

\addtocounter{section}{0}
\begin{document}

\title{ Graphs with the minimum spectral radius  for given independence number \footnote{Supported by National Natural Science Foundation of China (Nos. 11971274, 12061074, 11671344).}}
\author{
{\small Yarong Hu$^{a,b}$, \  \ Qiongxiang Huang$^{a,}$\footnote{Corresponding author.
\it Email addresses:  huangqx@xju.edu.cn (Q. Huang).},\ \ Zhenzhen Lou$^{a}$}\\[2mm]
\footnotesize $^a$ College of Mathematics and System Science,
Xinjiang University, Urumqi 830046, China \\
\footnotesize $^b$ School of Mathematics and Information Technology, Yuncheng University, Yuncheng 044000, China}
\date{}
\maketitle {\flushleft\large\bf Abstract}
Let $\mathbb{G}_{n,\alpha}$ be the set of  connected graphs with order $n$ and  independence number $\alpha$. Given $k=n-\alpha$, the graph with  minimum spectral radius among $\mathbb{G}_{n,\alpha}$ is called the minimizer  graph.
Stevanovi\'{c} in the classical book [D. Stevanovi\'{c}, Spectral Radius of Graphs, Academic Press, Amsterdam, 2015.] pointed that determining minimizer  graph in $\mathbb{G}_{n,\alpha}$ appears to be a tough problem on page $96$.
Very recently, Lou and Guo in \cite{Lou} proved that
 the minimizer  graph of $\mathbb{G}_{n,\alpha}$ must be a tree if $\alpha\ge\lceil\frac{n}{2}\rceil$.
In this paper,  we further give  the structural features for the  minimizer  graph  in detail, and then provide of a constructing theorem for  it. Thus, theoretically we completely determine the minimizer  graphs in $\mathbb{G}_{n,\alpha}$ along with their spectral radius for any given $k=n-\alpha\le \frac{n}{2}$.
As an application, we  determine all the minimizer  graphs  in $\mathbb{G}_{n,\alpha}$  for $\alpha=n-1,n-2,n-3,n-4,n-5,n-6$   along with their  spectral radii, the first four results are known in \cite{Xu,Lou} and the last two are new.
\begin{flushleft}
\end{flushleft}
\section{Introduction}\label{se-1}
A graph $G$ considered throughout this paper is  simple with  vertex set $V(G)$ and  edge set $E(G)$, where $n(G)=|V(G)|$ denote the \emph{order} and $|E(G)|=m(G)$ the \emph{size} of $G$.
The set of the neighbors of a vertex $v\in V(G)$ is denoted
by $N_{G}(v)$, and  the degree of $v$ is denoted by $d_G(v)$ (or  $d(v)$ for short).
For  $U\subseteq V(G)$, let $G[U]$ be the subgraph induced by $U$.
The \emph{adjacency matrix} of a  graph $G$ is defined as $A(G)=(a_{ij})$, where $a_{ij}=1$ if  vertices $i$ and $j$ are adjacent, and $a_{ij}=0$ otherwise.
The largest eigenvalue of $A(G)$ is called the \emph{spectral radius} of $G$, denoted by $\rho(G)$.
A subset $S\subseteq V(G)$ is called an \emph{independent set} of $G$ if there is no edge between every vertex in $S$.
The \emph{independence number} of $G$, denoted by $\alpha(G)$,  is the maximum cardinality of an independent set in $G$.
Let $\mathbb{G}_{n,\alpha}$ be the set of connected graphs with order $n$ and  independence number $\alpha$.
Through out this paper,
we call a graph is a \emph{minimizer graph} if it attains the minimum spectral radius over all graphs in $\mathbb{G}_{n,\alpha}$ for given $k=n-\alpha$.
As usual,
let $K_n$, $K_{1,n-1}$ and $P_n$ be  respectively the complete graph, the star and the path of order $n$.

The research on  the maximum or minimum spectral radius among connected graphs with a given graph invariant has been studied extensively.
The extremal graphs with the maximum  spectral radius for a given clique number, chromatic number, matching number, diameter, independence number and so on,  are studied in  \cite{111,61,153,72,143,Stevanovic,Lu}.
Compared to the maximum  spectral radius,
characterizing  the graphs with the minimum spectral radius is difficult.
The results for giving diameter is only determined for some special values.
For examples,  the graph  $G$ of the minimum spectral radius with small diameter  $diam(G)\in \{1, 2, 3, 4\}$  are determined in \cite{36,8}, with large diameter $diam(G)\in \{n-1,n-2,n-3, \lfloor n/2\rfloor\}$ in \cite{154},   $diam(G)=n-4$ in \cite{166},
simultaneously, $diam(G)\in \{n-4,n-5\}$ in \cite{37}, and $ diam(G)\in\{n-6,n-7,n-8\}\cup [\frac{n}{2}, \frac{2n-3}{3}]$  in\cite{87}.
In general, characterizing  the graphs of the minimum spectral radius is still an open problem \cite{Stevanovic}.
Particularly, Stevanovi\'{c} \cite{Stevanovic} pointed out that determining the  graph with the minimum spectral radius among connected graph with independence number  $\alpha$ appears  to be a tough problem.
In fact, the corresponding results for the independence number are less compared to the diameter, here we  list them in Tab.\ref{result} for references.
\begin{table}[h]
  \centering
 \footnotesize
\begin{tabular*}{13.5cm}{p{40pt}|p{25pt}|p{130pt}|p{25pt}|p{97pt}}
\hline
References&$\alpha$&The minimizer graph&Year&Author
 \\\hline
\cite{Xu}& $1$& the complete graph $K_n$&2009& Xu, Hong, Shu and Zhai\vspace{0.1cm}\\
 \cite{Xu}&$2$& $F(\lceil\frac{n}{2}\rceil,\lfloor\frac{n}{2}\rfloor)$ (see Fig.\ref{fig-W-D})& 2009&Xu, Hong, Shu and Zhai\vspace{0.1cm}\\
 \cite{Du}&$3$& $P(\frac{n}{3},3)$ for $3|n$, where $\frac{n}{3}\ge15$& 2013&Du and Shi\vspace{0.1cm}\\
 \cite{Du}&$4$& $P(\frac{n}{4},4)$ for $4|n$, where $\frac{n}{4}\ge24$& 2013&Du and Shi\vspace{0.1cm}\\ \hline
 \cite{Xu}&$\lceil\frac{n}{2}\rceil$& the path $P_n$&2009& Xu, Hong, Shu and Zhai\vspace{0.1cm}\\
 \cite{Xu}&$\lceil\frac{n}{2}\rceil+1$&
 $\left\{\begin{array}{ll}
 W_n& \mbox{ if $n$ is odd}\\
 D_n& \mbox{ if $n$ is even}
\end{array}\right.$ (see Fig.\ref{fig-W-D})&2009& Xu, Hong, Shu and Zhai\vspace{0.1cm}\\
\cite{Lou}& $n-4$& Theorem 1.2& 2022&Lou and Guo\vspace{0.1cm}\\
 \cite{Xu}&$n-3$&
Theorem 3.2
&2009& Xu, Hong, Shu and Zhai\vspace{0.1cm}\\
\cite{Xu}& $n-2$& $T(\lceil\frac{n-3}{2}\rceil,\lfloor\frac{n-3}{2}\rfloor)$ (see Fig.\ref{fig-W-D})& 2009&Xu, Hong, Shu and Zhai\vspace{0.1cm}\\
 \cite{Xu}&$n-1$& the star $K_{1,n-1}$& 2009&Xu, Hong, Shu and Zhai\\\hline
\end{tabular*}
  \caption{\small Some results}\label{result}
\end{table}
In  Tab.\ref{result} the first four lines list the results for  small independence number $\alpha\in \{1,2,3,4\}$, and the others for large $\alpha\in \{\lceil\frac{n}{2}\rceil, \lceil\frac{n}{2}\rceil+1, n-4, n-3,n-2,n-1\}$.
There leaves  a large unknown range for $\alpha$ to research.
Let $T^*$ be the minimizer graph of $\mathbb{G}_{n,\alpha}$ for given $k=n-\alpha$.

In this paper, we restrict on $\alpha\ge\lceil\frac{n}{2}\rceil$ to  characterize the minimizer graph $T^*$    and determine its spectral radius $\rho(T^*)$.
Recently, for $\alpha\ge\lceil\frac{n}{2}\rceil$, Lou and Guo \cite{Lou} gave a general result that the graph with  minimum spectral radius in $\mathbb{G}_{n,\alpha}$ is a tree.
Based on this result, we further give  the structural features for the  minimizer  graph, and then provide of a constructing theorem for  it.
Consequently we completely determine the minimizer  graphs in $\mathbb{G}_{n,\alpha}$ along with their spectral radii for any given $k=n-\alpha\le \frac{n}{2}$.

Our article is organized as follows.
In Section 2,  we give several necessary lemmas and concepts as well as notations.
In Section 3 we   characterize the structure of  the minimizer graph in $\mathbb{G}_{n,\alpha}$ for $k=n-\alpha\le \frac{n}{2}$   in detail (see Theorem \ref{thm-minimizer-main}).
In Section 4, we give  a construction theorem for the minimizer graphs in $\mathbb{G}_{n,\alpha}$ together with their  spectral radii for $k=n-\alpha\le \frac{n}{2}$ ( see Theorem \ref{thm-minimizer-spectral-radius}).
In Section 5, as an application, we  determine  the minimizer  graphs  in $\mathbb{G}_{n,\alpha}$  for $\alpha=n-1,n-2,n-3,n-4,n-5,n-6$   along with their spectral radii, the first four results are known in \cite{Xu,Lou} and the last two are new.

\section{Preliminaries}
In the section, we cite some useful lemmas and notations for the later use.

\begin{lem}[\cite{Cvetkovic1}]\label{lem-subgraph-radius}
If $H$ is the subgraph of  a connected graph $G$, then $\rho(H)\le \rho(G)$. Particularly, if $H$ is proper then $\rho(H)< \rho(G)$.
\end{lem}
Let $G$ be a connected graph. By the Perron-Frobenius theorem, there is unique positive unit eigenvector of $A(G)$ corresponding to  $\rho(G)$, which is called the \emph{ Perron vector} of $G$, and denoted by $\mathbf{x}$,
the entry of $\mathbf{x}$ corresponding to vertex $u$ is denoted by $x_u$.

\begin{lem}[\cite{Wu}]\label{lem-perron-entry-radius}
Let $u$, $v$ be two distinct vertices of a connected graph $G$. Suppose  $w_1,w_2,...,w_t$ $(t\ge1)$ are some vertices of $N_G(v)\setminus N_G(u)$ and $\mathbf{x}$ is the  Perron vector of $G$. Let $G'=G-\{vw_i\mid i=1,2,...,t\}+\{uw_i\mid i=1,2,...,t\}$.  If $x_u\ge x_v$ then $\rho(G)<\rho(G')$.
\end{lem}

An \emph{internal path} of $G$ is a sequence of  vertices $v_1,v_2,...,v_s$ with $s\ge 2$ such that:\\
(i) the vertices in the sequence are distinct (except possibly $v_1=v_s$),\\
(ii) $v_i$ is adjacent to $v_{i+1}$ ($i=1,2,...,s-1$),\\
(iii) the vertex degrees satisfy $d(v_1)\ge3$, $d(v_2)=\cdots=d(v_{s-1})=2$ (unless $s=2$) and $d(v_s)\ge3$.
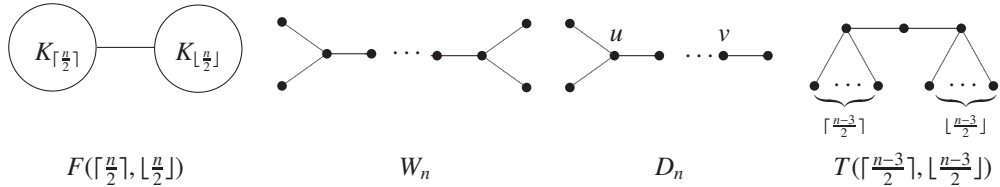
\begin{figure}[h]
\centering
\unitlength 0.85mm % = 2.845pt
\linethickness{0.4pt}
\footnotesize
\ifx\plotpoint\undefined\newsavebox{\plotpoint}\fi % GNUPLOT compatibility
\begin{picture}(154,27.831)(0,0)
\put(5.831,22){\circle{13.662}}
\put(12.831,22){\line(1,0){9}}
\put(28.662,21.831){\circle{13.662}}
\put(3,20){$K_{\lceil\frac{n}{2}\rceil}$}
\put(25,20){$K_{\lfloor\frac{n}{2}\rfloor}$}
\put(41.75,26){\circle*{1.5}}
%\emline(41.75,26)(48.75,21)
\multiput(41.75,26)(.046979866,-.033557047){149}{\line(1,0){.046979866}}
%\end
%\emline(48.75,21)(41.75,16)
\multiput(48.75,21)(-.046979866,-.033557047){149}{\line(-1,0){.046979866}}
%\end
\put(41.75,16){\circle*{1.5}}
\put(48.75,21){\line(1,0){7}}
\put(54.75,21){\line(1,0){1}}
\put(48.75,21){\circle*{1.5}}
\put(55.75,21){\circle*{1.5}}
\put(59,21){$\ldots$}
\put(65,20.75){\line(1,0){1}}
\put(66,20.75){\circle*{1.5}}
\put(66,20.75){\line(1,0){7}}
\put(73,20.75){\circle*{1.5}}
%\emline(73,20.75)(80,25.75)
\multiput(73,20.75)(.046979866,.033557047){149}{\line(1,0){.046979866}}
%\end
\put(80,25.75){\circle*{1.5}}
%\emline(73,20.75)(80,15.75)
\multiput(73,20.75)(.046979866,-.033557047){149}{\line(1,0){.046979866}}
%\end
\put(80,15.75){\circle*{1.5}}
\put(86.75,25.75){\circle*{1.5}}
%\emline(86.75,25.75)(93.75,20.75)
\multiput(86.75,25.75)(.046979866,-.033557047){149}{\line(1,0){.046979866}}
%\end
%\emline(93.75,20.75)(86.75,15.75)
\multiput(93.75,20.75)(-.046979866,-.033557047){149}{\line(-1,0){.046979866}}
%\end
\put(86.75,15.75){\circle*{1.5}}
\put(93.75,20.75){\line(1,0){7}}
\put(99.75,20.75){\line(1,0){1}}
\put(93.75,20.75){\circle*{1.5}}
\put(100.75,20.75){\circle*{1.5}}
\put(104.75,20.5){$\ldots$}
\put(110.75,20.75){\line(1,0){7}}
\put(116.75,20.75){\line(1,0){1}}
\put(110.75,20.75){\circle*{1.5}}
\put(117.75,20.75){\circle*{1.5}}
\put(130,25){\circle*{1.5}}
\put(139,25){\circle*{1.5}}
\put(148,25){\circle*{1.5}}
%\emline(130,25)(135,16)
\multiput(130,25)(.033557047,-.060402685){149}{\line(0,-1){.060402685}}
%\end
\put(135,16){\circle*{1.5}}
%\emline(148,25)(143,16)
\multiput(148,25)(-.033557047,-.060402685){149}{\line(0,-1){.060402685}}
%\end
%\emline(130,25)(125,16)
\multiput(130,25)(-.033557047,-.060402685){149}{\line(0,-1){.060402685}}
%\end
\put(125,16){\circle*{1.5}}
\put(143,16){\circle*{1.5}}
%\emline(148,25)(153,16)
\multiput(148,25)(.033557047,-.060402685){149}{\line(0,-1){.060402685}}
%\end
\put(153,16){\circle*{1.5}}
\put(130,25){\line(1,0){9}}
\put(139,25){\line(1,0){9}}
\put(128,16){$\ldots$}
\put(146,16){$\ldots$}
\put(125,15){$\underbrace{}_{\lceil\frac{n-3}{2}\rceil}$}
\put(144,15){$\underbrace{}_{\lfloor\frac{n-3}{2}\rfloor}$}
\put(93,23){$u$}
\put(110,23){$v$}
\put(100,2){$D_n$}
\put(60,2){$W_n$}
\put(8,2){$F(\lceil\frac{n}{2}\rceil,\lfloor\frac{n}{2}\rfloor)$}
\put(128,2){$T(\lceil\frac{n-3}{2}\rceil,\lfloor\frac{n-3}{2}\rfloor)$}
\end{picture}
  \caption{\footnotesize{Some minimizer graphs}}\label{fig-W-D}
\end{figure}

\begin{lem}[\cite{Hoffman}]\label{lem-subdividing-internal-radius}
Suppose that $G\not= W_n$ (see Fig.\ref{fig-W-D}) and $uv$ is an edge on  an internal path of $G$. Let $G_{uv}$ be the graph obtained from $G$ by the subdivision of  the edge $uv$. Then $\rho(G_{uv})<\rho(G)$.
\end{lem}

\begin{lem}[\cite{eigenspace}]\label{vector}
Let $G$ be a connected graph with the adjacent matrix $A(G)$ and let the vector $\mathbf{0}<\mathbf{y}\in \mathbb{R}^{n(G)}$ (i.e., any entries of $\mathbf{y}$ is non-negative and $\mathbf{y}\not=\mathbf{0}$). If $A(G)\mathbf{y}\le \lambda\mathbf{y}$ then $\rho(G)\le \lambda$.
\end{lem}

Denote by $L(T)$ the set of  all leaves of  a tree $T\not=P_{2t}$,
and let $L(P_{2t})=\{u\}$, where  $u$ is  one end vertex of $P_{2t}$.
\begin{lem}[\cite{Lu}]
 For every tree $T$ , there exists a maximum independent set $S(T)$ of $T$ such that
$L(T )\subseteq S(T )$.
\end{lem}
 Later we always assume that $S(T)$ is  the maximum independent set of  $T$  such that $L(T )\subseteq S(T )$.

%Denote by $diam(G)$ the diameter of a connected graph $G$.
\begin{lem}[\cite{Lou}]\label{lem-diam}
Let $G$ be a connected graph with order $n$ and $\alpha(G)\ge \lceil \frac{n}{2}\rceil$. Then $diam(G)\le 2(n-\alpha(G))$.
\end{lem}

Let $G=(A,B)$ be a bipartite graph with $A=\{u_1,\ldots,u_t\}$ and let $G'=G\circ (l(u_1),\ldots,l(u_t))$ be a graph obtained from $G$ by joining $u_i$  with $l(u_i)\ge0$ new pendant vertices for
$1\le i \le t$.
Particularly, we denote  $G'=G\circ  l\mathbf{1}_{A}$ if $l(u_1)=\cdots=l(u_t)=l$ is a constant.
The following lemma is crucial for our main result.
\begin{lem}[\cite{Csikvari}]\label{lem-bipartite}
Let $G=(A,B)$ be a bipartite graph  and let $G'=G\circ l\mathbf{1}_{A}$.  Then $\rho(G')=\sqrt{\rho^2(G)+l}$.
\end{lem}

\section{The structure of the minimizer graph in $\mathbb{G}_{n,\alpha}$ }
In this section, we first cite several lemmas in \cite{Lou} that describe  the structure and properties for minimizer graph, and then we provide a series of propositions which together give some new characters of the minimizer graph's structure.  We first quote a nice result from \cite{Lou} which determines the shape of the minimizer graph.
\begin{lem}[\cite{Lou}, Theorem 1.2]\label{thm-extremal-graph-tree}
Let $G$ be the minimizer graph in $\mathbb{G}_{n,\alpha}$, where  $\alpha\ge \lceil\frac{n}{2}\rceil$. Then $G$ is a tree.
\end{lem}

A vertex $v$ of a tree $T\not=P_{n},D_n$ (see Fig.\ref{fig-W-D}) is called a \emph{control vertex} if $d(v)\ge3$  and $T-v$ contains at least $d(v)-1$  paths.
Particularly, $v_2$ and $v_{2t}$ are defined as the control vertices for $P_{2t+1}=v_1v_2\cdots v_{2t}v_{2t+1}$,
$v_2$ and $v_{2t}$  the control vertices for $P_{2t}=v_1v_2\cdots v_{2t}$,
$v$ and $u$  the control vertices for $D_{n}$ shown in Fig.\ref{fig-W-D}.
Additionally, denote $L(P_{2t})=\{v_1\}$.
A path $P$ of $T$ is called \emph{control path} if its two ends  are control vertices.
Particularly, if a tree $T$ contains  only one control vertex, then the  control path  of $T$ is $P_1$.
 It is clear that  a star $K_{1,n-1}$, $P_{n}$ and $D_{n}$ have unique control path as well as $W_n$, and the diameter path of $T$ must contain a proper control path.

Under the assumption of $\lceil\frac{n}{2}\rceil+2\le \alpha(T)\le n-2$, the authors defined in \cite{Lou} the concepts that a \emph{branching vertex} $u\in T$ is a vertex with $d(u)\ge 3$ and an   \emph{end branching vertex} is such a branching vertex $u$  which  does not lie on any path between other two branching vertices. Obviously, an end branching vertex must be  control vertex, but not vice versa. The control vertex extends the concept of end branching vertex to $P_{n}$, $D_n$ and $K_{1,n-1}$  because of $\alpha(P_{n})=\lceil\frac{n}{2}\rceil$, $\alpha(D_{n})=\lceil\frac{n}{2}\rceil+1$ and $\alpha(K_{1,n-1})=n-1$. The authors in \cite{Lou} gave three propositions about end branching vertex, and here we summarize them in the following lemma.

\begin{lem}[\cite{Lou}, Propositions 3.1, 3.2 and 3.3]\label{pro-branching-vertices}
Let $T$ be  a minimizer  graph in  $\mathbb{G}_{n,\alpha}$, where $\lceil\frac{n}{2}\rceil+2\le \alpha\le n-2$. Then \\
 (i) $T$ has at least two end branching vertices,\\
 (ii)  Every end branching vertex only  attaches some  leaves (by another words,  end branching vertex $u$ attaches exactly $d(u)-1$  leaves),\\
 (iii) Let $v_0$, $v_q$ be two end branching vertices of $T$ and $P_{q+1}=v_0v_1\cdots v_q$  be a path that connects $v_0$ and $v_q$. Then $q$ is even and $v_i\in S(T)$ for odd $i\in[1, q-1]$.
\end{lem}

Lemma \ref{pro-branching-vertices} can be extended as the following lemma.
\begin{lem}\label{pro-end-branching}
Let $T^*$ be a minimizer  graph in  $\mathbb{G}_{n,\alpha}$ for $ \alpha\ge \lceil\frac{n}{2}\rceil$, and  $P_{q+1}=v_0v_1\cdots v_q$ be a control path of $T^*$. Then\\
(i)  Every control vertex $u$ of $T^*$  attaches at least $d(u)-1$  leaves,\\
(ii) $q$ is even and $v_i\in S(T^*)$ for odd $i\in[1, q-1]$.
\end{lem}

\begin{proof}
From Lemma \ref{pro-branching-vertices}, (i) and (ii) hold for $\lceil\frac{n}{2}\rceil+2\le \alpha\le n-2$. It remains to consider  $\alpha=n-1$, $\lceil\frac{n}{2}\rceil$ and $\lceil\frac{n}{2}\rceil+1$. According to the results listed Table \ref{result},  if $\alpha(T^*)\in \{n-1,\lceil\frac{n}{2}\rceil,\lceil\frac{n}{2}\rceil+1\}$ then $T^* $ would be one of  $K_{1,n-1}$, $P_n$, $D_n$ or $W_n$.  From Fig.\ref{fig-W-D} we see that  either one of their control vertices can only hang leaves and thus (i) holds.
Additionally, the  control path of $K_{1,n-1}$, $P_n$, $D_n$ and $W_n$ are all unique by definition.
It is easy to verify that   (ii) holds for $T^*\in \{K_{1,n-1}, P_n, D_n, W_n\}$.
\end{proof}

In what follows we always denote by $T^*$  a minimizer  graph in  $\mathbb{G}_{n,\alpha}$ for $ \alpha\ge \lceil\frac{n}{2}\rceil$ as assumed as in Lemma \ref{pro-end-branching}.

\begin{lem}\label{pro-path-1}
Let $v'v$ be an edge of a tree $T$, and $v$ attach with a  pendant path $P=vy_1y_2\cdots y_s$ for $s\ge 2$.
 For even $t\le s$, let $T'$ be a graph obtained from $T$ by replacing  $v'v$ with a path $v'x_1x_2\cdots x_tv$  and simultaneously deleting $t$ vertices $y_{s-t+1},y_{s-t+2},...,y_s$. If  one of $v$ and $v'$ is in $S(T)$, then $n(T')=n(T)$ and $\alpha(T')=\alpha(T)$.
\end{lem}
\begin{proof}
Let $T_1$ be the graph obtained from $T$ by replacing  $v'v$ with a path $v'x_1x_2\cdots x_tv$ and $T_2$  a graph obtained from $T_1$ by deleting vertices $y_{s-t+1},y_{s-t+2},...,y_s$. Then $T'=T_2$ and, obviously $n(T)=n(T')$.

We first verify $\alpha(T_1)=\alpha(T)+\frac{t}{2}$.
Without loss of generality, assume that $v\in S(T)$ and then $v'\not\in S(T)$.
Thus $S(T)\cup\{x_1,x_3,...,x_{t-1}\}$ is an independent set of $T_1$ and so
 $\alpha(T_1)\ge\alpha(T)+\frac{t}{2}$.
 For the maximum independent set $S(T_1)$, let $S_1=S(T_1)\cap \{x_1,x_2,...,x_t\}$ and $S_2=S(T_1)\cap V(T)$.
  Then $\alpha(T_1)=|S(T_1)|=|S_1|+|S_2|$.
  Since $x_1x_2\cdots x_t$ is an induced path in $T_1$, it can produce at most $\frac{t}{2}$ independent vertices and so $|S_1|\le \frac{t}{2}$.
  If $|S_2|\le |S(T)|$ then $\alpha(T_1)=|S(T_1)|\le \alpha(T)+\frac{t}{2}$. Otherwise $|S_2|> |S(T)|$, it implies  that $v', v\in S_2$ ( because  if there is at most one of  $v'$ and $ v$   belonging  to $S_2$ then $S_2$ will be an independent set of $T$ which contracts $|S_2|> |S(T)|$ ).
In this situation,  we have $|S_2|=\alpha(T)+1$ and $|S_1|\le \frac{t}{2}-1$, and so $\alpha(T_1)=|S_1|+|S_2|\le \alpha(T)+\frac{t}{2}$. It follows that $\alpha(T_1)=\alpha(T)+\frac{t}{2}$.

Next, by considering whether $y_{s-t}\in S(T_1)$ or not, we can verify $\alpha(T_2)=\alpha(T_1)-\frac{t}{2}$. In fact, $s-t$ has the same  parity with $s$, if  $y_{s-t}\in S(T_1)$ then $y_{s-t+2}, y_{s-t+4},...,y_s\in S(T_1)$; if $y_{s-t}\not\in S(T_1)$ then $y_{s-t+1}, y_{s-t+3},...,y_{s-1}\in S(T_1)$. It leads to our conclusion.
\end{proof}

\begin{pro}\label{pro-path}
Let $P_{q+1}=v_0v_1\cdots v_q$  be a control path of $T^*$. Then  $v_i$ does not attach  pendant paths of length more than one for $ 1\le i \le q-1$.
Moreover,  $v_i$ does not attach any leaf for odd $i$.
\end{pro}

\begin{proof}
Obviously, our  result holds if $T^*=P_n$. It is clear that $T^*\not=P_n$ has a control vertex of degree at least $3$.  Without loss of generality, we assume that $v_0$ is the control vertex with $d(v_0)\ge3$.
Suppose to the contrary that $v_i$ attaches a pendant path $P =v_iy_1\cdots y_s$ with $s\ge2$.  For an even $2\le t\le s$, let $T_1$ be a graph obtained from $T^*$ by replacing  $v_{i-1}v_i$ with a path $v_{i-1}x_1x_2\cdots x_tv_i$. Note that $d_{T^*}(v_0), d_{T^*}(v_i)\ge 3$, the edge $v_{i-1}v_i$ belongs to an internal path of $T^*$ that is some sub-path of $P_{q+1}$, thus we have  $\rho(T_1)<\rho(T^*)$ from
Lemma \ref{lem-subdividing-internal-radius}.
Let $T_2$ be  a graph obtained from $T_1$ by deleting vertices $y_{s-t+1},y_{s-t+2},...,y_s$. Notice that  there is exactly one of $v_{i-1}$ and $v_i$ belonging to $S(T^*)$ according to Lemma \ref{pro-end-branching}(ii), we have $n(T_2)=n(T^*)$ and $\alpha(T_2)=\alpha(T^*)$ by Lemma \ref{pro-path-1}. Hence $T_2\in \mathbb{G}_{n,\alpha}$.
Since $T_2$ is a subgraph of $T_1$, we have $\rho(T_2)<\rho(T_1)$ from Lemma \ref{lem-subgraph-radius},
 and thus $\rho(T_2)<\rho(T^*)$ which contracts the  minimality  of $T^*$.

Particularly, we have $v_i\in S(T^*)$ for odd $i$ by Lemma \ref{pro-end-branching}(ii). It implies that
 $v_i$ does not attach any leaf  for odd $i$.
\end{proof}

\begin{lem}\label{split}
Let $G$ be a  connected graph with the  Perron vector $\mathbf{x}=(x_u\mid u\in V(G))$ and $v\in V(G)$.
Suppose that $N_G(v)=\{w_1,w_2,...,w_t\}$ ($t\ge3$) and $x_{w_1}=\min_{w\in N_G(v)} \{x_w\}$.
Let $G'$ be a graph obtained from $G$ by replacing $v$ with  two new vertices $v'$,  $v''$ and meanwhile adding new edges  $v'w_1, v'w_2,...,v'w_s$  and  $\ v''w_1, v''w_{s+1},\ v''w_{s+2},...,v''w_t$ for some $s \in[2, t-1]$ (see Fig.\ref{graph transformation-v}).
If  $vw_1$ is  a cut edge of $G$ then $\rho(G')\le \rho(G)$,
with equality if and only if
$t=3$ and $x_{w_1}=x_{w_2}=x_{w_3}$.
\end{lem}
\begin{figure}[h]
  \centering
  \footnotesize
\unitlength 1mm % = 2.845pt
\linethickness{0.4pt}
\ifx\plotpoint\undefined\newsavebox{\plotpoint}\fi % GNUPLOT compatibility
\begin{picture}(128,28)(0,0)
\put(23,26){\circle*{1.5}}
\put(25,26){$v$}
\put(19,17){\circle*{1.5}}
\put(19,17){\line(0,-1){5}}
\put(18.5,8){$\vdots$}
\put(20,15){$w_2$}
\put(25,0){\footnotesize$G$}
\put(38,17){\circle*{1.5}}
\put(38,17){\line(0,-1){5}}
\put(37.5,8){$\vdots$}
\put(39,15){$w_{s+1}$}
\put(41,12){$\ldots$}
\put(23,26){\line(5,-3){15}}
\put(28.75,17){\circle*{1.5}}
\put(28.75,17){\line(0,-1){5}}
\put(28.5,8){$\vdots$}
\put(29.75,15){$w_s$}
\put(23,26){\line(2,-3){6}}
\put(47,17){\circle*{1.5}}
\put(47,17){\line(0,-1){5}}
\put(46.5,8){$\vdots$}
\put(48,15){$w_t$}
\put(22,12){$\ldots$}
\put(9,17){\circle*{1.5}}
\put(9,17){\line(0,-1){5}}
\put(8.5,8){$\vdots$}
\put(10,15){$w_1$}
\put(9,12){\oval(6,14)[]}
%\emline(23,26)(9,17)
\multiput(23,26)(-.0524344569,-.0337078652){267}{\line(-1,0){.0524344569}}
%\end
%\emline(23,26)(19,17)
\multiput(23,26)(-.033613445,-.075630252){119}{\line(0,-1){.075630252}}
%\end
%\emline(23,26)(47,17)
\multiput(23,26)(.0898876404,-.0337078652){267}{\line(1,0){.0898876404}}
%\end
\put(34,16.5){\oval(36,23)[]}
\put(0,10){\footnotesize$G_{w_1}$}
\put(53,10){\footnotesize$G_{v}$}
\put(99,26){\circle*{1.5}}
\put(100,25){$v'$}
\put(95,17){\circle*{1.5}}
\put(95,17){\line(0,-1){5}}
\put(94.5,8){$\vdots$}
\put(96,15){$w_2$}
\put(114,17){\circle*{1.5}}
\put(114,17){\line(0,-1){5}}
\put(113.5,8){$\vdots$}
\put(115,15){$w_{s+1}$}
\put(117,12){$\ldots$}
\put(104.75,17){\circle*{1.5}}
\put(104.75,17){\line(0,-1){5}}
\put(104.5,8){$\vdots$}
\put(105.75,15){$w_s$}
\put(99,26){\line(2,-3){6}}
\put(123,17){\circle*{1.5}}
\put(123,17){\line(0,-1){5}}
\put(122.5,8){$\vdots$}
\put(124,15){$w_t$}
\put(97,12){$\ldots$}
\put(85,17){\circle*{1.5}}
\put(85,17){\line(0,-1){5}}
\put(84.5,8){$\vdots$}
\put(86,15){$w_1$}
\put(85,12){\oval(6,14)[]}
%\emline(99,26)(85,17)
\multiput(99,26)(-.0524344569,-.0337078652){267}{\line(-1,0){.0524344569}}
%\end
%\emline(99,26)(95,17)
\multiput(99,26)(-.033613445,-.075630252){119}{\line(0,-1){.075630252}}
%\end
\put(116,26){\circle*{1.5}}
\put(117,25){$v''$}
%\emline(116,26)(114,17)
\multiput(116,26)(-.03333333,-.15){60}{\line(0,-1){.15}}
%\end
%\emline(116,26)(123,17)
\multiput(116,26)(.033653846,-.043269231){208}{\line(0,-1){.043269231}}
%\end
%\emline(116,26)(85,17)
\multiput(116,26)(-.1161048689,-.0337078652){267}{\line(-1,0){.1161048689}}
%\end
\put(110,16.5){\oval(36,23)[]}
\put(76,9){\footnotesize$G'_{w_1}$}
\put(129,9){\footnotesize$G'_{v',v''}$}
\put(100,0){\footnotesize$G'$}
\end{picture}

\caption{\footnotesize{Graph $G$ and $G'$ in Lemma \ref{split} }}
  \label{graph transformation-v}
\end{figure}
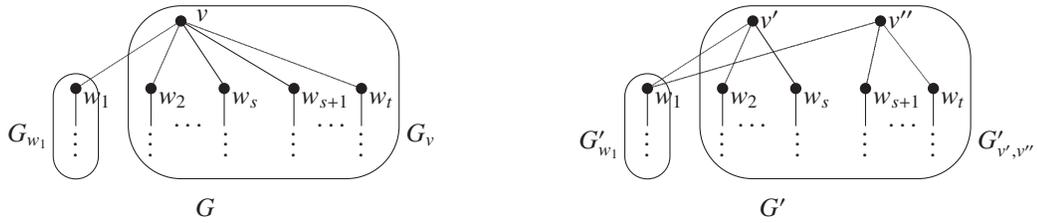

\begin{proof}
Since $vw_1$ is  a cut edge,  let $G_v$ and $G_{w_1}$ be two components of $G-vw_1$ containing $v$  and $w_1$, respectively.
Similarly, denote by $G'_{w_1}$  the component of $G'-v'w_1-v''w_1$ containing $w_1$, and  $G'_{v',v''}$ the remaining part containing $v'$ and $v''$.
It is clear that $G'_{v',v''}-\{v',v''\}=G_v-v$ and $G'_{w_1}=G_{w_1}$.
Now we define a vector $\mathbf{y}=(y_u\mid u\in V(G'))$, where
$$y_u=
\left\{\begin{array}{ll}
x_v &\mbox{if $u\in \{v',v''\}$,}\\
x_u &\mbox{if $u\in G'_{v',v''}-\{v',v''\}$,}\\
2x_u &\mbox{if $u\in G'_{w_1}$.}
\end{array}\right.$$
Clearly, $\mathbf{y}> \mathbf{0}$.
In what follows  we will consider the $u$-entry of $A(G')\mathbf{y}$. First of all, from our definition   $(A(G')\mathbf{y})_u=\rho(G)y_u$ if $u\not=v',v'',w_1$.
For $u=v'$, we have
\begin{equation}\label{eq-v'}
\begin{array}{ll}
(A(G')\mathbf{y})_{v'}=\sum_{i=1}^s y_{w_i}&=2x_{w_1}+\sum_{i=2}^s x_{w_i}\\
&=\sum_{i=1}^t x_{w_i}+(x_{w_1}-\sum_{i=s+1}^t x_{w_i})\\
&=\rho(G)x_v+(x_{w_1}-\sum_{i=s+1}^t x_{w_i})\\
&\le \rho(G)x_v\ \ \ \ ( \mbox{due to $x_{w_1}=\min_{w\in N_G(v)} \{x_w\}$} )\\
&=\rho(G)y_{v'}.
\end{array}\end{equation}
Similarly,  we have
\begin{equation}\label{eq-v''}
(A(G')\mathbf{y})_{v''}=y_{w_1}+\sum_{i=s+1}^t y_{w_i}=\rho(G)x_v+(x_{w_1}-\sum_{i=2}^s x_{w_i})\le \rho(G)x_v= \rho(G)y_{v''}.\end{equation}
Let $N_{G}(w_1)=\{v,u_1,...,u_p\}$ and then $N_{G'}(w_1)=\{v',v'',u_1,...,u_p\}$. We have
$$
(A(G')\mathbf{y})_{w_1}=y_{v'}+y_{v''}+\sum_{i=1}^p y_{u_i}=x_v+x_v+\sum_{i=1}^p 2x_{u_i}
=2\rho(G)x_{w_1}=\rho(G)\cdot(2 x_{w_1})
=\rho(G)y_{w_1}.$$
By above discussions, we have $A(G')\mathbf{y}\le \rho(G)\mathbf{y}$.
By Lemma \ref{vector}, we have $\rho(G')\le \rho(G)$.

The  equalities (\ref{eq-v'}) and (\ref{eq-v''}) hold (i.e., $A(G')\mathbf{y}=\rho(G)\mathbf{y}$) if and only if $t=3$ and $x_{w_1}=x_{w_2}=x_{w_3}$.
Note that $G'$ is connected and the vector $\mathbf{y}$ is positive,   $\mathbf{y}$ is an eigenvector of $A(G')$ corresponding to the spectral radius $\rho(G')$ if $A(G')\mathbf{y}=\rho(G)\mathbf{y}$.
Thus
$\rho(G')= \rho(G)$ if and only if $t=3$ and $x_{w_1}=x_{w_2}=x_{w_3}$.

It completes the proof.
\end{proof}

\begin{pro}\label{pro-middle-branch-vertex}
Let $P_{q+1}=v_0v_1\cdots v_q$  be a control path of $T^*$.
Then $d_{T^*}(v_i)=2$  for odd $i\in[1, q-1]$.
\end{pro}

\begin{proof}
Suppose to the contrary that $d_{T^*}(v_i)\ge3$  for some odd $i\in[1, q-1]$.
Then we can label  $N_{T^*}(v_i)=\{ w_{1},\ldots, w_{s} ,v_{i-1},v_{i+1}\}$ with  $s\ge1$.
Let $\mathbf{x}$ be the  Perron vector of $T^*$.
Without loss of generality, we may assume that $x_{w_1}=\min_{w\in N_{T^*}(v_i)} \{x_w\}$.
Let $T_2$ be a graph obtained from $T^*$ by replacing $v_i$ with  two new vertices $v'_i$, $v''_i$ and meanwhile adding  edges $v'_iw_1,v'_iw_2,...,v'_iw_s$ and $ v''_iw_1,v_i''v_{i-1}, v''_iv_{i+1}$.
 Clearly, $n(T_2)=n(T^*)+1$.
 By  Lemma \ref{pro-end-branching}(ii), we have $v_i\in S(T^*)$ for odd $i\in[1, q-1]$.
Thus $v_{i-1},v_{i+1}\notin S(T^*)$ and $w_t\notin S(T^*)$ for $t\in [1,s]$.
It follows that
  $S(T_2)=(S(T^*)\setminus \{v_i\})\cup \{v'_i,v''_i\}$ and so $\alpha(T_2)=\alpha(T^*)+1$.
Hence $T_2\in \mathbb{G}_{n+1,\alpha+1}$.
Since $v_{i}w_1$ is a cut edge of $T^*$, we have $\rho(T_2)\le \rho(T^*)$
by Lemma \ref{split}.
Let $T_3$ be a graph obtained from $T_2$ by deleting a leaf. Thus $T_3\in \mathbb{G}_{n,\alpha}$. By Lemma \ref{lem-subgraph-radius}, we have  $\rho(T_3)< \rho(T_2)\le\rho(T^*)$, a contradiction.

It completes the proof.
\end{proof}

\begin{pro}\label{cor-diam}
$diam(T^*)$ is even expect of $T^*=P_n$. Moreover, we have $diam(T^*)=2,4$ and $n-1$ for $\alpha=n-1, n-2$ and $\lceil\frac{n}{2}\rceil$, respectively, and  $6\le diam(T^*)\le 2(n-\alpha)$ for $\alpha\in [\lceil\frac{n}{2}\rceil+1, n-3]$.
\end{pro}

\begin{proof}
First of all, from Tab.\ref{result}, we know that $T^*=K_{1,n-1}$ if $\alpha=n-1$, and so $diam(T^*)=2$. If $\alpha=n-2$ then $T^*=T(\lceil\frac{n-3}{2}\rceil,\lfloor\frac{n-3}{2}\rfloor)$, and so $diam(T^*)=4$ (see Fig.\ref{fig-W-D}). If $\alpha=\lceil\frac{n}{2}\rceil$ then $T^*=P_n$, and so $diam(T^*)=n-1$. Next we suppose that $\alpha\in [\lceil\frac{n}{2}\rceil+1, n-3]$.
Clearly, there is  a diameter path $P=u_0u_1\cdots u_{D-1}u_D$ that connects the two leaves $u_0$ and $u_{D}$ of $T^*$, which, we know,  includes  a control path $P'=u_1\cdots u_{D-1}$, where $D=diam(T^*)$.
 By Lemma  \ref{pro-end-branching}(ii), $D-2$ is even   and so $D$ is even.

First we have $diam(T^*)\le 2(n-\alpha)$ by Lemma \ref{lem-diam}.
On the other hand, note that the diameter path $P$ contains two control vertices, we have $diam(T^*)\ge4$.
It suffices to verify $diam(T^*)\not=4$.
Suppose to the contrary,
 the diameter path becomes  $P=u_0u_1u_2u_3u_4$.
 Thus  $u_1$ and $u_3$ are two control vertices. It implies that $u_0$, $u_4\in S(T^*)$, moreover $u_2 \in S(T^*)$ by Lemma \ref{pro-end-branching}(ii) and $u_1,u_3\notin S(T^*)$.
  Recall that $\alpha\le n-3$, we have $|V(T^*)\setminus S(T^*)|=n-\alpha\ge 3$. Thus, apart from $u_1$ and $u_3$  there exists at least one neighbor of $u_2$, say  $w\notin S(T^*)$. Thus  $d_{T^*}(u_2)\ge3$,  it   contradicts  Proposition \ref{pro-middle-branch-vertex}.
\end{proof}

For each $x\in V(T^*)\setminus L(T^*)$, according to Lemma \ref{pro-end-branching}(i) and Proposition \ref{pro-path}, there exists a control path $P= v_0v_1\cdots v_q$ (not necessarily unique)  such that $x$ is a vertex of $P$, i.e., $x=v_i$ for some $0\le i\le q$.
According to Lemma \ref{pro-end-branching}(ii)  and  Proposition \ref{pro-middle-branch-vertex} we have known that
$$
\left\{\begin{array}{ll}
\mbox{ $q$  is  even}&\\
\mbox{ $v_i\notin S(T^*)$}& \mbox {if   $i$ is even}\\
\mbox{ $v_i\in S(T^*)$ and $d_{T^*}(v_i)=2$ } & \mbox {if   $i$ is odd}\\
\end{array}\right.
$$
It is clear that the labelling of $x=v_i$ depends on the choice of $P$, however the  parity of its subscript $i$ is  independent with $P$.  So we can define such a vertex $x=v_i$  \emph{odd } if $i$ is odd, and  \emph{even } otherwise. Thus $V(T^*)\setminus L(T^*)$ is partitioned as even vertices, denoted by $V_2^*$, and odd vertices, denoted by $V_1^*$.
 It is clear that $S(T^*)=L(T^*)\cup V_1^*$ and so $V_1^*$ is independent. Moreover, any two  vertices $x_1$ and $x_2$ in $V_2^*$ cannot be adjacent since otherwise edge $x_1x_2$ must belong to some control path, thus $V_2^*$ is also independent. It is clear that $V_1^*\cup V_2^*$ induces a subtree of $T^*$, i.e., $T^*[V_1^*\cup V_2^*]=T^*\setminus L(T^*)$,  which is called  the \emph{main tree} of the minimizer graph $T^*$ and denoted by $T^*_-$.
 Notice that the leaves of $T^*$ can only hang at some vertices of $V_2^*$ according to Proposition \ref{pro-path},  the leaves of $T^*_-$ must be the control vertices of $T^*$, and the degree of vertex in $V_1^*$  is unchanged in $T^*_-$.  We immediately have the following proposition.
\begin{pro}\label{eq-degree1}
$T^*_{-}=T^*\setminus L(T^*)$ is  a tree  with bipartite partition $(V_1^* ,V_2^*)$, and
$$d_{T^*_{-}}(v)=\left\{\begin{array}{ll}
d_{T^*}(v)=2&  \mbox{ if $v\in V^*_1$},\\
d_{T^*}(v)-l(v)& \mbox{ if $v\in V^*_2$},
\end{array}\right.$$
 where $l(v)$ is the number of the leaves in $T^*$ attached at $v$.
\end{pro}

\begin{pro}\label{lem-main-tree}
 $|V^*_2|=n-\alpha$, $|V^*_1|=n-\alpha-1$, $n(T^{*}_{-})=2(n-\alpha)-1$ and $|L(T^*)|=2\alpha-n+1$.
\end{pro}
\begin{proof}
According to above discussions, we have $V(T^*)=L(T^*)\cup V^*_1\cup V^*_2$ and  $S(T^*)=L(T^*)\cup V^*_1$ are two partitions,  which gives $|V^*_2|=|V(T^*)|-|S(T^*)|=n-\alpha$.
Since $T^{*}_{-}=( V^*_1, V^*_2)$ is bipartite, we have $m(T^{*}_{-})=\sum_{v\in V^*_1} d_{T^*_{-}}(v) =2|V^*_1|$ by Proposition \ref{eq-degree1}. Note that $T^{*}_{-}$ is a tree, we have
\begin{equation}\label{eq-n-1}
n(T^{*}_{-})=m(T^{*}_{-})+1=2|V^*_1|+1.\end{equation}
On the other hand, it is clear that
\begin{equation}\label{eq-n-2}
n(T^{*}_{-})=|V^*_2|+|V^*_1|=n-\alpha+|V^*_1|.\end{equation}
Combining (\ref{eq-n-1}) and (\ref{eq-n-2}), we obtain
$|V^*_1|=n-\alpha-1$ and $n(T^{*}_{-})=2(n-\alpha)-1$. Thus $|L(T^*)|=|S(T^*)|- |V^*_1|=\alpha-(n-\alpha-1)=2\alpha-n+1$.
\end{proof}

For $\alpha=\lceil\frac{n}{2}\rceil$, we have $T^*=P_{n}$. Note that $T^*_{-}=T^*\setminus L(T^*)$. By the choice of $L(P_{n})$,  $T^*_-=P_{n-2}$ for odd $n$ and $T^*_-=P_{n-1}$ for even $n$. Thus $diam(T^{*}_{-})=2(n-\lceil\frac{n}{2}\rceil)-2$ is even. Moreover,
from Lemma  \ref{pro-end-branching}(i) and  Proposition \ref{cor-diam},
we have the following result.
\begin{pro}\label{main-tree-diam}
$diam(T^{*}_{-})$ is even, moreover $diam(T^*_{-})=0$ and $2$ for $\alpha=n-1$ and $n-2$, respectively, and $4 \le diam(T^*_{-}) \le2(n-\alpha)-2$ for $\alpha\in [\lceil\frac{n}{2}\rceil, n-3]$.
\end{pro}

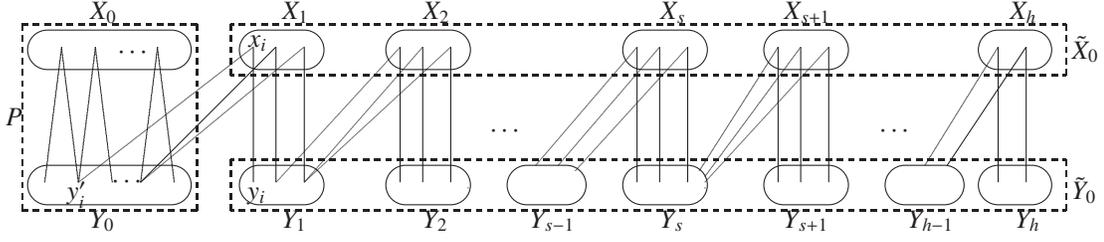
\begin{figure}[h]
  \centering
  \footnotesize
\unitlength 0.75mm % = 2.845pt

\linethickness{0.4pt}
\ifx\plotpoint\undefined\newsavebox{\plotpoint}\fi % GNUPLOT compatibility
\begin{picture}(188,37)(0,0)
%\emline(10,32)(7,8)
\multiput(10,32)(-.03370787,-.26966292){89}{\line(0,-1){.26966292}}
%\end
%\emline(10,32)(13,8)
\multiput(10,32)(.03370787,-.26966292){89}{\line(0,-1){.26966292}}
%\end
%\emline(16,32)(13,8)
\multiput(16,32)(-.03370787,-.26966292){89}{\line(0,-1){.26966292}}
%\end
%\emline(16,32)(19,8)
\multiput(16,32)(.03370787,-.26966292){89}{\line(0,-1){.26966292}}
%\end
%\emline(27,32)(24,8)
\multiput(27,32)(-.03370787,-.26966292){89}{\line(0,-1){.26966292}}
%\end
%\emline(27,32)(30,8)
\multiput(27,32)(.03370787,-.26966292){89}{\line(0,-1){.26966292}}
%\end
\put(18.5,31.5){\oval(29,7)[]}
\put(18.5,7.5){\oval(29,7)[]}
\put(49,31.5){\oval(15,7)[]}
\put(49,37){$X_{1}$}
\put(74,37){$X_{2}$}
\put(75,31.5){\oval(15,7)[]}
\put(43,32){$x_{i}$}
\put(49,7.5){\oval(15,7)[]}
\put(82,7){\line(0,1){0}}
\put(49,0){$Y_{1}$}
\put(74,0){$Y_{2}$}
\put(75,7.5){\oval(15,7)[]}
\put(43,5){$y_{i}$}
\put(3,3){\dashbox{1}(31,33)[cc]{}}
\put(15,0){$Y_0$}
\put(0,18){$P$}
\put(11,5){$y'_{i}$}
\put(15,37){$X_0$}
\put(44,32){\line(0,-1){24}}
\put(53,32){\line(0,-1){24}}
\put(48,32){\line(0,-1){24}}
\put(70,32){\line(0,-1){24}}
\put(79,32){\line(0,-1){24}}
\put(74,32){\line(0,-1){24}}
%\emline(70,32)(48,8)
\multiput(70,32)(-.0336906585,-.0367534456){653}{\line(0,-1){.0367534456}}
%\end
%\emline(74,32)(53,8)
\multiput(74,32)(-.0337078652,-.0385232745){623}{\line(0,-1){.0385232745}}
%\end
%\emline(79,32)(53,8)
\multiput(79,32)(-.03651685393,-.03370786517){712}{\line(-1,0){.03651685393}}
%\end
%\emline(44,32)(13,8)
\multiput(44,32)(-.04353932584,-.03370786517){712}{\line(-1,0){.04353932584}}
%\end
%\emline(53,32)(24,8)
\multiput(53,32)(-.04073033708,-.03370786517){712}{\line(-1,0){.04073033708}}
%\end
\put(20,30){$\cdots$}
\put(19,7){$\cdots$}
\put(48,32){\line(-1,-1){24}}
\put(86,16){$\cdots$}
\put(116,37){$X_{s}$}
\put(117,31.5){\oval(15,7)[]}
\put(142,31.5){\oval(15,7)[]}
\put(138,37){$X_{s+1}$}
\put(124,7){\line(0,1){0}}
\put(116,0){$Y_{s}$}
\put(117,7.5){\oval(15,7)[]}
\put(142,7.5){\oval(15,7)[]}
\put(93,0){$Y_{s-1}$}
\put(138,0){$Y_{s+1}$}
\put(112,32){\line(0,-1){24}}
\put(121,32){\line(0,-1){24}}
\put(116,32){\line(0,-1){24}}
\put(137,32){\line(0,-1){24}}
\put(146,32){\line(0,-1){24}}
\put(141,32){\line(0,-1){24}}
%\emline(137,32)(123,10)
\multiput(137,32)(-.0337349398,-.0530120482){415}{\line(0,-1){.0530120482}}
%\end
%\emline(141,32)(124,9)
\multiput(141,32)(-.0337301587,-.0456349206){504}{\line(0,-1){.0456349206}}
%\end
%\emline(146,32)(124,8)
\multiput(146,32)(-.0336906585,-.0367534456){653}{\line(0,-1){.0367534456}}
%\end
\put(96,7.5){\oval(14,7)[]}
%\emline(112,32)(94,11)
\multiput(112,32)(-.0337078652,-.0393258427){534}{\line(0,-1){.0393258427}}
%\end
%\emline(116,32)(98,11)
\multiput(116,32)(-.0337078652,-.0393258427){534}{\line(0,-1){.0393258427}}
%\end
%\emline(121,32)(101,10)
\multiput(121,32)(-.0337268128,-.0370994941){593}{\line(0,-1){.0370994941}}
%\end
\put(179,31.5){\oval(13,7)[]}
\put(178,37){$X_{h}$}
\put(179,7.5){\oval(13,7)[]}
\put(160,0){$Y_{h-1}$}
\put(179,0){$Y_{h}$}
\put(176,32){\line(0,-1){24}}
\put(181,32){\line(0,-1){24}}
%\emline(176,32)(163,11)
\multiput(176,32)(-.0336787565,-.0544041451){386}{\line(0,-1){.0544041451}}
%\end
\put(181,32){\line(-2,-3){14}}
\put(155,16){$\cdots$}
\put(163,7.5){\oval(14,7)[]}
\put(40,27){\dashbox{1}(148,9)[cc]{}}
\put(40,3){\dashbox{1}(148,9)[cc]{}}
\put(189,30){$\tilde{X}_0$}
\put(189,5){$\tilde{Y}_0$}
\end{picture}

  \caption{\footnotesize The graph $T^*_{-}$}\label{main}
\end{figure}

Let $P=v_0v_1\cdots v_q$ be a control path of $T^{*}$,
which is also a path in main tree $T^*_-$ connecting two leaves $v_0$ and $v_q$.
Let $X_0$ and $Y_0$ be respectively the sets of odd and even vertices in $P$ and clearly $T^{*}_{-}[X_0\cup Y_0]=P$.
Let $\tilde{X}_0= V^*_1\setminus X_0$ and $\tilde{Y}_0= V^*_2\setminus Y_0$.
The induced subgraph $H=T^{*}_{-}[\tilde{X}_0\cup \tilde{Y}_0]=T^{*}_{-}\setminus P$ is described in  Fig.\ref{main}.
By  Proposition \ref{eq-degree1}, odd vertex $v\in X_0$ can only join two even vertices on  the path $P$, we have $E_{T^*_{-}}(X_0,\tilde{Y}_0)=\emptyset$ and thus $N_{T^*_{-}}(\tilde{Y}_0)\subseteq \tilde{X}_0$ because $T^*_{-}=(V_1^* ,V_2^*)$ is bipartite (see Fig.\ref{main}).

Now we consider the induced subgraph $H=T^{*}_{-}[\tilde{X}_0\cup \tilde{Y}_0]$.
Recall that $d_{T^{*}_{-}}(x)=2$ for any odd vertex $x\in \tilde{X}_0$, its two  neighbors $y$ and $y'$ belong to $V_2^*=Y_0\cup \tilde{Y}_0$.
Obviously, $y$ and $y'$ can not simultaneously belong to $Y_0$, since otherwise $T^*_{-}$ will form a cycle containing $x$, $y$ and $y'$.
Denote by $X_{1} $ the subset of $\tilde{X}_0$ consisting of the vertex $x\in \tilde{X}_0$ such that its one neighbor $y'\in Y_0$ and another $y\in \tilde{Y}_0$.
Then  $\tilde{X}_0\setminus X_{1}$ consists of vertices in $\tilde{X}_0$ whose two neighbors are in $\tilde{Y}_0$ (see Fig.\ref{main}).

Let $X_{1}=\{x_{1},x_{2},...,x_{t}\}$ and $Y_{1}=N_H(X_{1})\subseteq \tilde{Y}_{0}$ (see Fig.\ref{main}). Clearly, $x_{i}$ has two neighbors $y'_{i}\in Y_0$ and $y_{i}\in Y_1$.
For $i\not=j$, we see  that $y_{i}\not=y_{j}$ since otherwise $T^*_-$ has a cycle containing $x_{i}$  and $x_{j}$. It implies that $|X_{1}|=|Y_{1}|$ (see Fig.\ref{main}).
In addition, the $t$ vertices of $X_{1}$ belong to distinct components of $H$ since otherwise the component, say containing $x_{1}$ and $x_{2}$, will join at the vertices of $P$ and form a cycle in $T^*_{-}$. Thus $H$ has at least $t$ components. Next
we  recursively  define $X_i$ and $Y_i$ for $i=2,3,...$ (refer to Fig.\ref{main}):
$$\begin{array}{ll}
&X_2=N_H(Y_1)\setminus X_1\subseteq \tilde{X}_0, \ \ Y_2=N_H(X_2)\setminus Y_1\subseteq \tilde{Y}_0,\  \ ...\\
&X_j=N_H(Y_{j-1})\setminus X_{j-1}\subseteq \tilde{X}_0,\ \ Y_j=N_H(X_j)\setminus Y_{j-1}\subseteq \tilde{Y}_0, \ \ ...
\end{array}$$
Note that the leaves of  $T^*_{-}$  are the control vertices of $T^*$ and included in $\tilde{Y}_0$. Each $Y_j\subset \tilde{Y}_0$ would contain some leaves of  $T^*_{-}$.  Without loss of generality, we can assume that  $Y_h$ contains only  leaves of $T^*_-$ (see Fig.\ref{main}).
It follows that
$\tilde{X}_0=\bigcup^h_{s=1}X_{s} \mbox{ and } \tilde{Y}_0=\bigcup^h_{s=1}Y_{s}$
are two partitions.
Since $T^*_{-}$ is connected and $X_1$ is a cut set, for any vertex  $ v\in\tilde{X}_0\cup \tilde{Y}_0$ there exists exactly one vertex $x \in X_1$ such that $v$ and $x$ are connected by one path in $H$. Thus,   $H$ has $|X_{1}|$ components. Moreover, for $1\le s \le h$ we know that each $x\in X_{s} $  has degree $d_H(x)=2$ and its one neighbor $y'\in Y_{s-1}$ and another $y\in Y_s$ (see Fig.\ref{main}). Such  neighbors $y'$ and $y$ are uniquely determined by $x$ and thus it induces two maps $\sigma': x\longrightarrow y'=\sigma'(x)$ and $\sigma: x\longrightarrow y=\sigma(x)$ for $x\in \tilde{X}_{0}$.  It is easy to see that $\sigma$ is a bijection, the induced subgraph $T^*_{-}[X_{s},Y_{s}]$  consists of  independent edges, denoted by $M_s=\{x\sigma(x)\mid x\in X_{s} \}$ and so  $|X_{s}|=|Y_{s}|$.
Thus, the edge set $E(T^*_{-})$ has the partition:
$$E(T^*_{-})=E(P)\ \bigcup\  (\cup^h_{s=1} M_s)\ \bigcup\
( \cup^h_{s=1} E(X_{s},Y_{s-1}) ),$$
where  $E(X_{s},Y_{s-1})=\{x\sigma'(x) \mid  x \in X_s\}$ for $s=1,..., h$.
Let $E'= \cup^h_{s=1} E(X_{s},Y_{s-1})=\{x\sigma'(x) \mid x\in \tilde{X}_{0}\}$.
 Then $|E'|=\sum^h_{s=1}|X_{s}|=|\tilde{X}_0|$.
Recall that $|X_0|=\frac{q}{2}$ and $|V^*_1|=n-\alpha-1$, we  have
$|\tilde{X}_0|=|V^*_1\setminus X_0|=n-\alpha-1-\frac{q}{2}$.
So $h\le \sum^h_{s=1}|X_{s}|=|\tilde{X}_0|=n-\alpha-1-\frac{q}{2}$.
On the other hand, for each $y_h\in Y_h$, we have $d_{T^*_-}(y_h,P)=2h$, where $d_{T^*_-}(y_h,P)$ is the minimum distant between $y_h$ and $v$ for any $v\in V(P)$.
Hence $d_{T^*_-}(y_h,P)+\frac{q}{2}=2h+\frac{q}{2}\le diam(T^*_-)$, and so
$h\le \lfloor\frac{diam(T^*_-)}{2}-\frac{q}{4}\rfloor$.
Summarizing the above discussion, we obtain the following results that are described by the above symbols.
\begin{pro}\label{lem-matching}
Let $P=v_0v_1\cdots v_q$ be a control path of $T^{*}$,
$X_0$ and $Y_0$ be the sets of odd and even vertices in $P$, respectively. Let %$\tilde{X}_0= V^*_1\setminus X_0$, $\tilde{Y}_0= V^*_2\setminus Y_0$,  and
$X_s$ and  $Y_s$ be defined above for $s=1, ..., h$
(see Fig.\ref{main}). Then\\
(i) $h\le \min\{  \ n-\alpha-1-\frac{q}{2}, \lfloor\frac{ diam(T^*_-)}{2}-\frac{q}{4}\rfloor \ \}$. Particularly,
if the control path $P$ of $T^{*}$
is a diameter path of $T^*_-$, then $h\le \min\{ \ n-\alpha-1-\frac{diam(T^*_-)}{2}, \lfloor\frac{ diam(T^*_-)}{4}\rfloor  \ \}$,\\
(ii) $V^*_1=\bigcup^h_{s=0} X_s$ and $V^*_2=\bigcup^h_{s=0} Y_s$ are two partitions,\\
 (iii) $T^*_{-}[X_0\cup Y_0]=P$, and $T^*_{-}[X_s\cup Y_s]=M_s$ is an edge independent set for $1\le s\le h$,  \\
 (iv) The edge set $E(T^*_{-})$ has the partition:
$E(T^*_{-})=E(P)\ \bigcup\  (\cup^h_{s=1} M_s)\ \bigcup\ \{x\sigma'(x) \mid x\in \cup^h_{s=1} X_s\}$.
Moreover, the number of edges in the three partitions is $q$, $n-\alpha-1-\frac{q}{2}$ and $n-\alpha-1-\frac{q}{2}$, respectively,\\
(v) The induced subgraph $H=T^*_{-}[(\cup^h_{s=1} X_s) \bigcup (\cup^h_{s=1} Y_s)]$  has the unique perfect matching  $\cup^h_{s=1} M_s$ and  $|X_1|$ components.
\end{pro}

Given a minimizer graph $T^*\in \mathbb{G}_{n,\alpha}$, usually its main tree $T^*_-$ is not unique in which we may take a control path
as   $P=v_0v_1\cdots v_d$ for $d=diam(T^*_-)$. By Proposition \ref{lem-matching}(iv), we know that $T^*_-$ is composed of $P$, $X_s$ and $Y_s$ along with the edges between them. To exactly, the  edges between $X_s$ and $Y_s$ is an edge independent set $M_s=\{x_{s1}y_{s1},x_{s2}y_{s2},...,x_{st_s}y_{st_s}\}$ determined by
 $\sigma:x_{si}\longleftrightarrow y_{si}$ for $i=1,2,...,t_s$. The edges between $X_s$ and $Y_{s-1}$ are labelled as $x_{sj}y_{s-1,j'}$ which is determined by $\sigma':x_{sj}\longrightarrow y_{s-1,j'}$.
 By using $\sigma'$ we now define a map $\sigma^*:M_s\longrightarrow M_{s-1}$ by   $\sigma^*(x_{sj}y_{sj})=x_{s-1,j'}y_{s-1,j'}$ ( if $s=1$ then we define $\sigma^*(x_{1j}y_{1j})=v_{s_j-1}v_{s_j}$ for some even vertex $v_0,v_d\not=v_{s_j}\in V(P)$ ).
 In this way, $\sigma^*$ determines main tree $T^*_-$, and thus  $T^*_-$  can be  written  by $T^*_-= T(d; M_1, M_2,...,M_h;\sigma^*)$.
By Proposition \ref{main-tree-diam}   and Proposition \ref{lem-matching}(i)(iv) we can  determine  the  parameters as in (\ref{eq-t})
\begin{equation}\label{eq-t}
\left\{\begin{array}{ll}
d=2(n-\alpha)-2 \mbox{ if $\alpha\in [ n-2, n-1]$}; 4\le  d\le 2(n-\alpha)-2 \mbox{ if $\alpha\in [\lceil\frac{n}{2}\rceil, n-3]$}\\
h\le \min\{ n-\alpha-1-\frac{d}{2}, \lfloor\frac{ d}{4}\rfloor \}\\
 \sum^h_{s=1} |M_s|=n-\alpha-1-\frac{d}{2}\\
\sigma^*:M_s\longrightarrow M_{s-1},\  1\le s \le h
  \end{array}\right.
  \end{equation}
All such main trees satisfying (\ref{eq-t}) we collect in the set  $\mathcal{T}^*_-(n,\alpha)$.
For example,
 let $\alpha=n-1$ ($n-2$ or $n-3$),  (\ref{eq-t}) becomes $d=0$ ($2$ or $4$), $h=0$ and $M_s=\emptyset$.
So $\mathcal{T}^*_-(n,n-1)=\{P_1\}$, $\mathcal{T}^*_-(n,n-2)=\{P_3\}$ and $\mathcal{T}^*_-(n,n-3)=\{P_5\}$ (see Fig.\ref{main-tree}),
which are just the results of   Theorems 3.1 and 3.2 in \cite{Xu}.
For $\alpha=n-4$,  (\ref{eq-t}) becomes
$$\left\{\begin{array}{ll}d\in \{4,6\}\\
h\le \min\{ 3-\frac{d}{2}, \lfloor\frac{ d}{4}\rfloor \}\\
\sum^h_{s=1} |M_s|=3-\frac{d}{2}\\
\sigma^*:M_s\longrightarrow M_{s-1},\  1\le s \le h
\end{array}\right.\Rightarrow
\left\{\begin{array}{ll}d=4\\
h= 1\\
\sum^h_{s=1} |M_s|=1\\
\sigma^*:M_s\longrightarrow M_{s-1},\   1\le s \le h
\end{array}\right.
\mbox{ or }
\left\{\begin{array}{ll}d=6\\
h=0\\
M_s=\emptyset
\end{array}\right.
$$
If the former occurs then $\sigma^*(x_{11}y_{11})=v_1v_2$ and we obtain a main tree $F_1=T(4; \{x_{11}y_{11}\}; \sigma^*)$; if the later occurs we obtain a main tree $F_2=T(6; \emptyset; \sigma^*)=P_7$. Thus
 $\mathcal{T}^*_-(n,n-4)=\{F_1,F_2\}$ (see Fig.\ref{main-tree}),
and it is just the result of   Theorem 1.3 in \cite{Lou}.

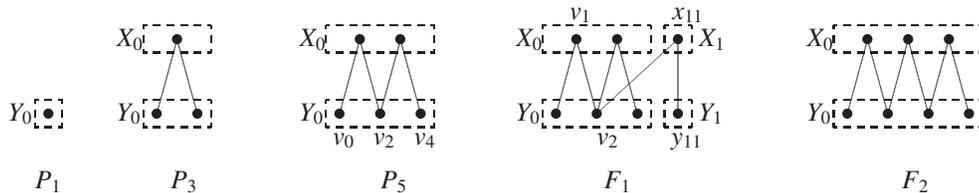
\begin{figure}[h]
  \centering
   \footnotesize
\unitlength 0.9mm % = 2.845pt
\linethickness{0.4pt}
\ifx\plotpoint\undefined\newsavebox{\plotpoint}\fi % GNUPLOT compatibility
\begin{picture}(146,24.5)(0,0)
\put(26,0){$P_3$}
\put(50,6.5){$v_0$}
\put(56,6.5){$v_2$}
\put(62,6.5){$v_4$}
\put(57,0){$P_5$}
\put(24,11){\circle*{1.5}}
\put(30,11){\circle*{1.5}}
%\emline(27,22)(24,11)
\multiput(27,22)(-.03370787,-.12359551){89}{\line(0,-1){.12359551}}
%\end
%\emline(27,22)(30,11)
\multiput(27,22)(.03370787,-.12359551){89}{\line(0,-1){.12359551}}
%\end
\put(27,22){\circle*{1.5}}
\put(51,11){\circle*{1.5}}
\put(57,11){\circle*{1.5}}
%\emline(54,22)(51,11)
\multiput(54,22)(-.03370787,-.12359551){89}{\line(0,-1){.12359551}}
%\end
%\emline(54,22)(57,11)
\multiput(54,22)(.03370787,-.12359551){89}{\line(0,-1){.12359551}}
%\end
\put(54,22){\circle*{1.5}}
\put(63,11){\circle*{1.5}}
%\emline(60,22)(57,11)
\multiput(60,22)(-.03370787,-.12359551){89}{\line(0,-1){.12359551}}
%\end
%\emline(60,22)(63,11)
\multiput(60,22)(.03370787,-.12359551){89}{\line(0,-1){.12359551}}
%\end
\put(60,22){\circle*{1.5}}
\put(22,9){\dashbox{1}(10,4)[cc]{}}
\put(22,20){\dashbox{1}(10,4)[cc]{}}
\put(49,9){\dashbox{1}(16,4)[cc]{}}
\put(49,20){\dashbox{1}(16,4)[cc]{}}
\put(45,21){$X_0$}
\put(45,10){$Y_0$}
\put(18,21){$X_0$}
\put(18,10){$Y_0$}
\put(2,10){$Y_0$}
\put(6,0){$P_1$}
\put(6,9){\dashbox{1}(4,4)[cc]{}}
\put(8,11){\circle*{1.5}}
\put(85,25){$v_{1}$}
\put(89,6.5){$v_{2}$}
\put(90,0){$F_1$}
\put(83,11){\circle*{1.5}}
\put(89,11){\circle*{1.5}}
%\emline(86,22)(83,11)
\multiput(86,22)(-.03370787,-.12359551){89}{\line(0,-1){.12359551}}
%\end
%\emline(86,22)(89,11)
\multiput(86,22)(.03370787,-.12359551){89}{\line(0,-1){.12359551}}
%\end
\put(86,22){\circle*{1.5}}
\put(95,11){\circle*{1.5}}
%\emline(92,22)(89,11)
\multiput(92,22)(-.03370787,-.12359551){89}{\line(0,-1){.12359551}}
%\end
%\emline(92,22)(95,11)
\multiput(92,22)(.03370787,-.12359551){89}{\line(0,-1){.12359551}}
%\end
\put(92,22){\circle*{1.5}}
\put(81,9){\dashbox{1}(16,4)[cc]{}}
\put(81,20){\dashbox{1}(16,4)[cc]{}}
\put(101,22){\line(0,-1){11}}
\put(101,22){\circle*{1.5}}
\put(101,11){\circle*{1.5}}
%\emline(101,22)(89,11)
\multiput(101,22)(-.0366972477,-.0336391437){327}{\line(-1,0){.0366972477}}
%\end
\put(99,20){\dashbox{1}(4,4)[cc]{}}
\put(99,9){\dashbox{1}(4,4)[cc]{}}
\put(104,21){$X_1$}
\put(104,10){$Y_1$}
\put(100,25){$x_{11}$}
\put(100,6.5){$y_{11}$}
\put(77,21){$X_0$}
\put(77,10){$Y_0$}
\put(134,0){$F_2$}
\put(126,11){\circle*{1.5}}
\put(132,11){\circle*{1.5}}
%\emline(129,22)(126,11)
\multiput(129,22)(-.03370787,-.12359551){89}{\line(0,-1){.12359551}}
%\end
%\emline(129,22)(132,11)
\multiput(129,22)(.03370787,-.12359551){89}{\line(0,-1){.12359551}}
%\end
\put(129,22){\circle*{1.5}}
\put(138,11){\circle*{1.5}}
%\emline(135,22)(132,11)
\multiput(135,22)(-.03370787,-.12359551){89}{\line(0,-1){.12359551}}
%\end
%\emline(135,22)(138,11)
\multiput(135,22)(.03370787,-.12359551){89}{\line(0,-1){.12359551}}
%\end
\put(135,22){\circle*{1.5}}
\put(144,11){\circle*{1.5}}
%\emline(141,22)(138,11)
\multiput(141,22)(-.03370787,-.12359551){89}{\line(0,-1){.12359551}}
%\end
%\emline(141,22)(144,11)
\multiput(141,22)(.03370787,-.12359551){89}{\line(0,-1){.12359551}}
%\end
\put(141,22){\circle*{1.5}}
\put(124,9){\dashbox{1}(22,4)[cc]{}}
\put(124,20){\dashbox{1}(22,4)[cc]{}}
\put(120,21){$X_0$}
\put(120,10){$Y_0$}
\end{picture}

  \caption{\footnotesize The main trees}\label{main-tree}
\end{figure}

Recall that  $T^*_-=T^*\setminus L(T^*)$ and  the leaves of $T^*$ can only hang at some even vertices of $T^*$, to exactly  $u\in V^*_2$ hangs $l(u)$ leaves (we specify $l(u)=0$ if $u$ hangs no any leaf).
Denote by $l_{V_2^*}=(l(u)\mid u\in V_2^*)$, we conclude that $T^*$ can be presented by $T^*=T^*_{-}\circ l_{V_2^*}$. It  naturally  follows the result.
\begin{thm}\label{thm-minimizer-main}
Let $T^*$ be the minimizer graph in $\mathbb{G}_{n,\alpha}$, where $\alpha\ge \lceil\frac{n}{2}\rceil$.  Then there exists $T^*_{-}\in \mathcal{T}^*_-(n,\alpha)$ such that
 $T^*=T^*_{-}\circ l_{V_2^*}$, where $T^*_{-}=(V_1^*,V_2^*)$ and $l_{V_2^*}=(l(u)\mid u\in V_2^*)$ satisfying $\sum_{u\in V_2^*}l(u)=|L(T^*)|=2\alpha-n+1$.
\end{thm}

Given $\alpha=n-k$ for  $1\le k\le \frac{n}{2}$, let $\mathbb{G}_{n,\alpha}^k=\{G\in \mathbb{G}_{n,\alpha}\mid n-\alpha=k\}$. For $T^*\in \mathbb{G}_{n,\alpha}^k$, Theorem \ref{thm-minimizer-main} tell us that $T^*=T^*_{-}\circ l_{V_2^*}$.  $T^*_{-}$ is usually  not unique and  determined by (\ref{eq-t1})
\begin{equation}\label{eq-t1}
\left\{\begin{array}{ll}
 d= 2k-2 \mbox{ if $k\in [1,2]$};
4\le  d\le 2k-2 \mbox{ if $k\in [3,\frac{n}{2}]$}\\
h\le \min\{k-1-\frac{d}{2}, \lfloor\frac{ d}{4}\rfloor \}\\
 \sum^h_{s=1} |M_s|=k-1-\frac{d}{2}\\
\sigma^*:M_s\longrightarrow M_{s-1},\  1\le s \le h&
  \end{array}\right.
  \end{equation}
We collect all such $T^*_{-}$ in the set  $\mathcal{T}^*_-(k)$. Moreover, by Proposition \ref{eq-degree1}, $T^*_{-}=(V_1^*,V_2^*)$  is bipartite ( refer to Fig.\ref{main}) satisfying $d_{T^*_{-}}(v)=d_{T^*}(v)=2$ for $v\in V_1^*$ and the leaves of $T^*_{-}$ are included in $V_2^*$ that are just the control vertices of $T^*$. Additionally, $|V_2^*|=k$ and $|V_1^*|=k-1$ by Proposition \ref{lem-main-tree}. The leaf sequence $l_{V_2^*}=(l(u)\mid u\in V_2^*)$ satisfies $\sum_{u\in V_2^*}l(u)=|L(T^*)|=2\alpha-n+1=n-2k+1$, which will be further characterized in detail in next section.

\section{Constructing Theorem for the minimizer graphs of $\mathbb{G}_{n,\alpha}^k$}
In this section, we will give  a constructing theorem for the minimizer graph  $T^*\in \mathbb{G}_{n,\alpha}^k$.  According to Theorem \ref{thm-minimizer-main}, $T^*$ comes from its main tree $T^*_-=(V_1^*,V_2^*)$ by adding a leaf sequence $l_{V_2^*}=(l(u)\mid u\in V_2^*)$ on $V_2^*$. Denote by $\bar{l}=\lfloor\frac{|L(T^*)|}{|V^*_2|}\rfloor=\lfloor\frac{n-2k+1}{k}\rfloor$ the average value  of $l_{V_2^*}$.
There exists a unique $0\le r \le k-1$ such that   $n-2k+1\equiv r$ ( $\mbox{mod } k$ ), i.e., $n+1\equiv r$ ( $\mbox{mod } k$ )  ( Note that henceforth we always adhere to this choice   for $r$ ), we have
$|L(T^*)|=n-2k+1=k\bar{l}+r$,
 $n(T^*)=n=2k-1+k\bar{l}+r$ and \begin{equation}\label{ll-eq}
\frac{n-3k+2}{k}\le
\bar{l}=\frac{|L(T^*)|-r}{k}=\frac{n-2k+1-r}{k}
\le\frac{n-2k+1}{k}.
\end{equation}

\begin{lem}\label{lem-bound}
For given $1\le k=n-\alpha\le \frac{n}{2}$, let $T^*$ be the minimizer graph in $\mathbb{G}^k_{n,\alpha}$. Then $\rho(T^*)<\sqrt{\bar{l}+5}$, where $\bar{l}=\lfloor\frac{n-2k+1}{k}\rfloor$.
\end{lem}

\begin{proof}
Let $P_{2k-1}$ be a path with bipartite partition $(X,Y)$ such that $X$ contains two end vertices of $P_{2k-1}$. Then $|X|=k$ and $|Y|=k-1$. We now construct a tree
$T=P_{2k-1}\circ (\bar{l}+1)\mathbf{1}_{X}$. Since $\rho(P_{2k-1})=2\cos\frac{\pi}{2k}<2$,
by Lemma \ref{lem-bipartite} we have $$\rho(T)=\sqrt{\rho^2(P_{2k-1})+\bar{l}+1}<\sqrt{\bar{l}+5}.$$
On the other aspect, recall that  $|L(T^*)|=k\bar{l}+r$, where $0\le r \le k-1$, we have $$|L(T)|=|X|(\bar{l}+1)=k(\bar{l}+1)=|L(T^*)|+k-r>|L(T^*)|.$$
Thus we can construct a subtree $T'$ by deleting $k-r$ leaves from $T$ such that $|L(T')|=|L(T^*)|$. Moreover,
$$\begin{array}{ll}
n(T')&=n(T)-(k-r)\\
&=|P_{2k-1}|+|X|(\bar{l}+1)-(k-r)\\
&=2k-1+ k\bar{l}+r\\&=n, \\
\alpha(T')&=\alpha(T)-(k-r)\\
&=|Y|+|X|(\bar{l}+1)-(k-r)\\
&=k-1 +k\bar{l}+r\\&=n-k.
\end{array}
$$
Thus $T'\in \mathbb{G}^k_{n,\alpha}$.
Since $T'$ is a subgraph of $T$, by Lemma \ref{lem-subgraph-radius} we have
$\rho(T^*)\le\rho(T')<\rho(T)<\sqrt{\bar{l}+5}$.

The proof is completed.
\end{proof}

By exchanging the position of $P_{2k-1}$ and $W_{2k+3}$ (see Fig.\ref{fig-W-D}) in the proof of Lemma \ref{lem-bound}, we get the following result.

\begin{lem}\label{cor-bound}
For given $1\le k=n-\alpha\le \frac{n}{2}$, let $T^*$ be the minimizer graph in $\mathbb{G}^k_{n,\alpha}$ and let $\bar{l}=\lfloor\frac{n-2k+1}{k}\rfloor$.
If $|L(T^*)|\le k\bar{l}+4$, then $\rho(T^*)\le \sqrt{\bar{l}+4}$.
\end{lem}
\begin{proof}
Let $W_{2k+3}=(X,Y)$ (see Fig.\ref{fig-W-D}) be a bipartite graph on $2k+3$ vertices such that $X$ does not contain any pendant vertices of $W_{2k+3}$.
Then $|X|=\lceil\frac{2k+3-4}{2}\rceil=k$.
Now we construct a tree $W=W_{2k+3}\circ \bar{l}\mathbf{1}_{X}$.
It's well known that
%$W_{2(n-\alpha)+3}$ is a Smith graph and
$\rho(W_{2k+3})=2$.
By Lemma \ref{lem-bipartite}, we have $$\rho(W)=\sqrt{\rho^2(W_{2k+3})+\bar{l}}=\sqrt{\bar{l}+4}.$$
Since $|L(T^*)|\le k\bar{l}+4$, from Proposition \ref{lem-main-tree} we have $$n(T^*)=n(T^*_{-})+|L(T^*)|\le 2k-1+k\bar{l}+4=n(W_{2k+3})+|X|\cdot  \bar{l}= n(W).$$
Then we can construct a subtree
  $W'$   by deleting $n(W)-n(T^*)$ leaves from $W$ such that $n(W')=n(T^*)=n$,
and so $$
\alpha(W')=\alpha(W)-(n(W)-n(T^*))=(n(W)-|X|)-(n(W)-n(T^*))=n-k.$$
Thus $W'\in \mathbb{G}^k_{n,\alpha}$.
Since $W'$ is a subtree of $W$, by Lemma \ref{lem-subgraph-radius} we obtain $\rho(T^*)\le\rho(W')\le \rho(W)\le\sqrt{\bar{l}+4}$.

The proof is completed.
\end{proof}

By Lemmas \ref{lem-bound} and \ref{cor-bound}, we can estimate the value of $l(u)$.

\begin{lem}\label{cor-s-range}
For given $1\le k=n-\alpha\le \frac{n}{2}$, let $T^*$ be the minimizer graph in $\mathbb{G}^k_{n,\alpha}$ and let $\bar{l}=\lfloor\frac{n-2k+1}{k}\rfloor$, $r=n-2k+1-\bar{l}k$.
For $u\in V^*_2$, we have
\begin{equation}\label{ll-eq-1}\bar{l}+r-2k+2-d_{T^*_-}(u) \le l(u)\le \bar{l}+4-d_{T^*_{-}}(u).\end{equation}
 Particularly, if  $0\le r\le 4$ then
\begin{equation}\label{ll-eq-2}\bar{l}+r-k+1-d_{T^*_-}(u)\le l(u)\le \bar{l}+3-d_{T^*_{-}}(u).\end{equation}
\end{lem}
\begin{proof}
By Lemma \ref{lem-bound}, we have $d_{T^*}(u)=\rho^2(K_{1,d_{T^*}(u)})< \rho^2(T^*)<\bar{l}+5$ and  for each $u \in V^*_2$,
\begin{equation}\label{ll-eq-3} l(u)=d_{T^*}(u)-d_{T^*_-}(u)\le \bar{l}+4-d_{T^*_-}(u).\end{equation}
On the other hand, note that  $|L(T^*)|=n-2k+1=k\bar{l}+r$, we  have
$$\begin{array}{ll}
 l(u)&=|L(T^*)|-\sum\limits_{u\not=v\in V^*_2}l(v)\\
 &\ge k\bar{l}+r-\sum\limits_{u\not=v\in V^*_2}(\bar{l}+4-d_{T^*_-}(v))\\
 &=k\bar{l}+r-(k-1)(\bar{l}+4)+\sum_{v\in V^*_2}d_{T^*_-}(v)-d_{T^*_-}(u)\\
 &=k\bar{l}+r-(k-1)(\bar{l}+4)+m(T^*_-)-d_{T^*_-}(u)\\
  &=k\bar{l}+r-(k-1)(\bar{l}+4)+2(k-1)-d_{T^*_-}(u)\\
 &=\bar{l}+r-2k+2-d_{T^*_-}(u)
\end{array}$$
which together with (\ref{ll-eq-3}) leads to (\ref{ll-eq-1}).

If  $0\le r\le 4$ then $|L(T^*)|\le k\bar{l}+4$. By Lemma \ref{cor-bound}, we have $d_{T^*}(u)=\rho^2(K_{1,d_{T^*}(u)})< \rho^2(T^*)\le\bar{l}+4$. Thus the corresponding  (\ref{ll-eq-3}) becomes \begin{equation}\label{eq-upper-1}
l(u)=d_{T^*}(u)-d_{T^*_-}(u)\le \bar{l}+3-d_{T^*_-}(u).\end{equation}
As the same arguments as above, from (\ref{eq-upper-1}) we get  (\ref{ll-eq-2}).

It completes the proof.
\end{proof}

According to Proposition \ref{lem-matching}(iii),  we have
\begin{equation}\label{eq-degree}
d_{T^*_-}(u)\le k-1 \ \mbox{ for $u\in V^*_2$}.\end{equation}
From (\ref{ll-eq-1}), (\ref{ll-eq}) and (\ref{eq-degree}) we have
\begin{equation}\label{eq-5-0}
\begin{array}{ll}
l(u)&\ge\bar{l}+r-2k+2-d_{T^*_-}(u)\\
&=\frac{n-2k+1-r}{k}+r-2k+2-d_{T^*_-}(u)\vspace{0.1cm}\\
&\ge \frac{n-2k+1-r}{k}+r-2k+2-(k-1)\vspace{0.1cm}\\
&=\frac{n-3k^2+k+ 1+(k-1)r}{k}.
\end{array}
\end{equation}
Now we define $n_0=3k^2-k-1-(k-1)r$
for a given $k\ge 1$
and
\begin{equation}\label{n-k}
\ell_{n,k}=\frac{n-3k^2+k+ 1+(k-1)r}{k}=\frac{n-n_0}{k}\ \mbox{  for $n\ge n_0$}.\end{equation}
Then we have $n_0+1=3k^2-k-kr+r$, and so $n_0+1\equiv r$ ( $\mbox{mod } k$ ).
Recall that $n+1\equiv r$ ( $\mbox{mod } k$ ), we have $n\equiv n_0$ ( $\mbox{mod } k$ ) and $\ell_{n,k}$ is an integer.
Based on the above   notions and symbols we give  a constructing theorem for the minimizer graph in $\mathbb{G}_{n,\alpha}^k$.
\begin{thm}\label{thm-minimizer-spectral-radius}
For  given  $1\le k=n-\alpha\le \frac{n}{2}$,
let  $n+1\equiv r$ ( $\mbox{mod } k$ ) and $n_0=3k^2-k-1-(k-1)r$, where $0\le r\le k-1$. For $n\ge n_0$,  $T^*$ is a  minimizer graph in $\mathbb{G}_{n,\alpha}^k$ if and only if there exists a minimizer graph $T_{n_0}\in \mathbb{G}_{n_0,n_0-k}^k$ such that $T^*=T_{n_0}\circ \ell_{n,k} \mathbf{1}_{V_2^*}$ and $\rho(T^*)=\sqrt{\rho^2(T_{n_0})+\ell_{n,k}}$, where $\ell_{n,k}$ is defined by (\ref{n-k}).
\end{thm}
\begin{proof}
Let $T^*$ be a  minimizer graph in $\mathbb{G}_{n,\alpha}^k$.
By Theorem \ref{thm-minimizer-main},
we have
 $T^*=T^*_{-}\circ l_{V_2^*}$ for some $T^*_{-}\in \mathcal{T}^*_-(k)$ and leaf sequence $l_{V_2^*}=(l(u)\mid u\in V_2^*)$ satisfying  $\sum_{u\in V_2^*}l(u)=n-2k+1$, moreover $l(u)$ satisfies  (\ref{ll-eq-1}) or (\ref{ll-eq-2}) by  Lemma \ref{cor-s-range}.
 Since $n\ge n_0$ and $k\ge 1$, we see that $0\le \ell_{n,k}\le l(u)$ for $u\in V_2^*$. Thus we can construct a graph $G$ obtained from $T^*$ by deleting $\ell_{n,k}$ leaves at each $u\in V_2^*$, i.e., $G=T^*_{-}\circ (l(u)-\ell_{n,k}\mid u\in V_2^*)$.
Note that $|V_2^*|=k$, we have $n(G)=n-\ell_{n,k}|V_2^*| =n_0$ and $\alpha(G)=\alpha-\ell_{n,k}|V_2^*|=\alpha-(n-n_0)=n_0-k$.
Thus $G\in \mathbb{G}_{n_0,n_0-k}^k$. On the other hand,
 it is clear that $T^*=G\circ \ell_{n,k}\mathbf{1}_{V_2^*}$.  By  Lemma \ref{lem-bipartite}, we have
\begin{equation}\label{mm-eq-1}
\rho(T^*)=\sqrt{\rho^2(G)+\ell_{n,k}}.\end{equation}
Notice that  $\ell_{n,k}$ is a determined number related with $n$ and $k$ and independent with the choice of $G$, from (\ref{mm-eq-1}) we see that $G $ must be a minimizer graph in  $\mathbb{G}_{n_0,n_0-k}^k$. The necessity holds.

Conversely, if there exists a minimizer graph $T_{n_0}\in \mathbb{G}_{n_0,n_0-k}^k$  such that $T^*=T_{n_0}\circ \ell_{n,k} \mathbf{1}_{V_2^*}$, then $n(T^*)=n_0+\ell_{n,k}|V_2^*|=n$ and $\alpha(T^*)=\alpha(T_{n_0})+\ell_{n,k}|V_2^*|=n_0-k+n-n_0=n-k=\alpha$. Thus $T^*\in\mathbb{G}_{n,\alpha}^k$ and  $\rho(T^*)=\sqrt{\rho^2(T_{n_0})+\ell_{n,k}}$ by  Lemma \ref{lem-bipartite}.
Since $T_{n_0}$ is a minimizer graph in $\mathbb{G}_{n_0,n_0-k}^k$ and $\ell_{n,k}$ is a determined number, $T^*$ is a minimizer graph in $\mathbb{G}_{n,\alpha}^k$. The sufficiency holds.
\end{proof}
\begin{remark}
According to Theorem \ref{thm-minimizer-spectral-radius}, the minimizer graph $T_{n_0}\in \mathbb{G}_{n_0,n_0-k}^k$ can be used to construct the minimizer graph of $\mathbb{G}_{n,\alpha}^k$. We call $T_{n_0}$ the \emph{kernel} of minimizer graph $T^*\in \mathbb{G}_{n,\alpha}^k$ if $T^*=T_{n_0}\circ \ell_{n,k} \mathbf{1}_{V_2^*}$. Given  $k$, $n_0=3k^2-k-1-(k-1)r$ is a constant, we can simply found the kernel of minimizer graph  among $\mathbb{G}_{n_0,n_0-k}^k$ by computer. Therefore, Theorem \ref{thm-minimizer-spectral-radius} completely determines the minimizer graph in $\mathbb{G}_{n,\alpha}^k$ and its spectral radius.
\end{remark}

\section{Construction for the minimizer  graphs in $\mathbb{G}^k_{n,\alpha}$ for $1\le k \le 6$}
Theorem \ref{thm-minimizer-spectral-radius} completely characterizes  the minimizer graphs of $\mathbb{G}_{n,\alpha}^k$ for $1\le k\le \frac{n}{2}$.
As an application,  we  will use a consistent approach  to find   the minimizer graphs of $\mathbb{G}^k_{n,\alpha}$  for $k= 1,2,3,4,5,6$. The  results for $k=1,2,3,4$ were made  in \cite{Xu,Lou} by different methods  and other two are new.

Given $1\le k\le  \frac{n}{2}$,
let  $n+1\equiv r$ ( $\mbox{mod } k$ ) and  $n\ge n_0=3k^2-k-1-(k-1)r$, where $0\le r\le k-1$ and $n+1\equiv n_0+1\equiv r$ ( $\mbox{mod } k$ ).
By Theorem \ref{thm-minimizer-spectral-radius}, each minimizer graph $T^*\in\mathbb{G}_{n,\alpha}^k$ can be presented by $T^*=T_{n_0}\circ \ell_{n,k} \mathbf{1}_{V_2^*}$, where the kernel $T_{n_0}$ is a minimizer graph in  $ \mathbb{G}_{n_0,n_0-k}^k$.
By Theorem \ref{thm-minimizer-main},  $F^*=T_{n_0}$ can be presented by $F^*=F^*_{-}\circ l_{V_2^*}$, where $F^*_{-}=(V_1^*,V_2^*)\in \mathcal{T}^*_-(k)$ and $l_{V_2^*}=(l(u)\mid u\in V_2^*)$ satisfying $\sum_{u\in V_2^*}l(u)=n_0-2k+1$.
 In addition, it follows from Lemma \ref{cor-s-range}  that $l_{V_2^*}$ is also restricted by (\ref{ll-eq-1}) or (\ref{ll-eq-2}).
Note that $n_0$ is small for $k=1,2,3,4$ and $\mathcal{T}^*_-(1)=\{P_1\}$, $\mathcal{T}^*_-(2)=\{P_3\}$,
$\mathcal{T}^*_-(3)=\{P_5\}$, $\mathcal{T}^*_-(4)=\{F_1,F_2\}$ (see Fig.\ref{main-tree}).
By using computer, we can simply choose the  minimizer graphs in $\mathbb{G}_{n_0,n_0-k}^k$ by selecting proper $l_{V_2^*}=(l(u)\mid u\in V_2^*)$  and finally get these minimizer graphs $T_{n_0}$ and then obtain $T^*\in\mathbb{G}_{n,\alpha}^k$ for $k=1,2,3$ and $4$.

For instance, let $k=3$, we have $n+1\equiv r$ ( $\mbox{mod } 3$ ) for some $0\le r\le2$, and  $n\ge n_0=23-2r$.
If $r=0$ then $n_0=23$.
By Theorem \ref{thm-minimizer-main}, $T_{n_0}=F^*=F^*_{-}\circ l_{V_2^*}$ is  a minimizer graph in $\mathbb{G}_{23,20}^3$, where $F^*_{-}=(V_1^*,V_2^*)\in \mathcal{T}^*_-(3)=\{P_5\}$ and $l_{V_2^*}=(l(v_0), l(v_2), l(v_4))$ satisfying $l(v_0)+ l(v_2)+l(v_4)=18$.
 Additionally, $l_{V_2^*}$ is also restricted by (\ref{ll-eq-2}).
By searching the proper $l_{V_2^*}$, we find that the minimizer graph achieves at $l_{V_2^*}=(l(v_0), l(v_2), l(v_4))=(7,4,7)$, i.e.,  $F^*=P_5\circ(7,4,7)$ is the minimizer graph in $\mathbb{G}_{23,20}^3$.
Note that $\ell_{n,k}=\frac{n-23}{3}$  from (\ref{n-k}),  by Theorem \ref{thm-minimizer-spectral-radius} we obtain the minimizer graph
$$\begin{array}{ll}T^*&=T_{n_0}\circ \ell_{n,k}\mathbf{1}_{V_2^*}\vspace{0.1cm}\\&= F^*\circ (\frac{n-23}{3},\frac{n-23}{3},\frac{n-23}{3})\vspace{0.1cm}\\
&=P_5\circ(7+\frac{n-23}{3},4+\frac{n-23}{3},7+\frac{n-23}{3})\vspace{0.1cm}\\
&=P_5\circ(\frac{n-2}{3},\frac{n-11}{3},\frac{n-2}{3})
\end{array}$$
in $\mathbb{G}_{n,\alpha}^3$ if  $n+1\equiv 0$ ( $\mbox{mod } 3$ ) (see the fourth line in Tab.\ref{table-min-1-4}). Similarly we can find the minimizer graphs in $\mathbb{G}_{n,\alpha}^3$ for  $n+1\equiv 1,2 $ ( $\mbox{mod } 3$ ) (see the fifth, sixth lines in Tab.\ref{table-min-1-4}).

As the same as above we can find  all the minimizer graphs $T^*\in\mathbb{G}_{n,\alpha}^k$ along with their parameters ( including spectral radius ), which are listed in Tab.\ref{table-min-1-4} for $k=1,2,3,4$, respectively. They are just all the minimizer graphs obtained in  \cite{Xu,Lou}.

\begin{table}[H]
\footnotesize
\caption{\small  The  minimizer graphs \\\footnotesize ( $n+1\equiv r$ ( $\mbox{mod } k$ ) and  $n\ge n_0=3k^2-k-1-(k-1)r$ )}
\centering
\renewcommand{\arraystretch}{1.3}
\begin{tabular*}{14.7cm}{p{2pt}|p{3pt}|p{7pt}|p{95pt}|p{28pt}|p{19pt}|p{125pt}|p{40pt}}
\hline
  $k$&$r$&$n_0$& the kernel $T_{n_0}\in \mathbb{G}_{n_0,n_0-k}^k$ &$\rho^2(T_{n_0})$&$\ell_{n,k}$& the minimizer graph $T^*\in \mathbb{G}_{n,\alpha}^k$
  & $\rho(T^*)$\\\hline
 $1$ &$0$&$1$&$P_{1}$ &$0$&$n-1$ &$P_1\circ (n-1)=K_{1,n-1}$&$\sqrt{n-1}$\\\hline
 \multirow{2}*{$2$}&$0$&$9$&
$P_3\circ(3,3)$&$5$
&$\frac{n-9}{2}$&
$P_3\circ(\frac{n-3}{2},\frac{n-3}{2})$
&$\sqrt{\frac{n+1}{2}}$\\\cline{2-8}
&$1$&$8$&$P_3\circ(2,3)$&
$\frac{7+\sqrt{5}}{2}$& $\frac{n-8}{2}$& $P_3\circ(\frac{n-4}{2},\frac{n-2}{2})$ &$\sqrt{\frac{n-1+\sqrt{5}}{2}}$\\\hline
\multirow{3}*{$3$}&$0$&$23$&
$P_5\circ(7,4,7)$& $7+\sqrt{3}$&$\frac{n-23}{3}$&
$P_5\circ(\frac{n-2}{3},\frac{n-11}{3},\frac{n-2}{3})$ &$\sqrt{\frac{n-2+3\sqrt{3}}{3}}$\\\cline{2-8}
&$1$&$21$&$P_5\circ(6,4,6)$&$8$&$\frac{n-21}{3}$& $P_5\circ(\frac{n-3}{3},\frac{n-9}{3},\frac{n-3}{3})$ &$\sqrt{\frac{n+3}{3}}$\\\cline{2-8}
&$2$&$19$&$P_5\circ(5,4,5)$&$6+\sqrt{2}$&$\frac{n-19}{3}$&
$P_5\circ(\frac{n-4}{3},\frac{n-7}{3},\frac{n-4}{3})$ &$\sqrt{\frac{n-1+3\sqrt{2}}{3}}$\\\hline
\multirow{7}*{$4$}&\multirow{2}*{$0$}&\multirow{2}*{$43$}&
$F_1\circ(10,6,10,10)$& \multirow{2}*{$12$}&
\multirow{2}*{$\frac{n-43}{4}$}&
$F_1\circ(\frac{n-3}{4},\frac{n-19}{4},\frac{n-3}{4},\frac{n-3}{4})$
 &
\multirow{2}*{$\sqrt{\frac{n+5}{4}}$}\\
&&&
$F_2\circ(10,8,8,10)$&&&$F_2\circ(\frac{n-3}{4},\frac{n-11}{4},\frac{n-11}{4},\frac{n-3}{4})$ &
\\\cline{2-8}
 &$1$&$40$&$F_1\circ(9,6,9,9)$&
$\frac{19+\sqrt{13}}{2}$&$\frac{n-40}{4}$&
$F_1\circ(\frac{n-4}{4},\frac{n-16}{4},\frac{n-4}{4},\frac{n-4}{4})$ &$\sqrt{\frac{n-2+2\sqrt{13}}{4}}$\\\cline{2-8}
&\multirow{3}*{$2$}&\multirow{3}*{$37$}&
$F_2\circ(8,7,7,8)$&\multirow{3}*{$\frac{19+\sqrt{5}}{2}$}
&\multirow{3}*{$\frac{n-37}{4}$}&
$F_2\circ(\frac{n-5}{4},\frac{n-9}{4},\frac{n-9}{4},\frac{n-5}{4})$
 &\multirow{3}*{$\sqrt{\frac{n+1+2\sqrt{5}}{4}}$}\\
&&&$F_2\circ(9,6,7,8)$&&&$F_2\circ(\frac{n-1}{4},\frac{n-13}{4},\frac{n-9}{4},\frac{n-5}{4})$&\vspace{0.1cm}\\
&&&$F_2\circ(9,6,6,9)$&&&
$F_2\circ(\frac{n-1}{4},\frac{n-13}{4},\frac{n-13}{4},\frac{n-1}{4})$ &\\ \cline{2-8}
 &$3$&$34$&$F_1\circ(8,3,8,8)$&$\frac{15+\sqrt{21}}{2}$
 &$\frac{n-34}{4}$&$F_1\circ(\frac{n-2}{4},\frac{n-22}{4},\frac{n-2}{4},\frac{n-2}{4})$ &$\sqrt{\frac{n-4+2\sqrt{21}}{4}}$\\
\hline
\end{tabular*}\label{table-min-1-4}
\end{table}
\begin{remark}
From Tab.\ref{table-min-1-4}, one can see that for $k=4$ the minimizer graph  is not unique if  $n+1\equiv r$ ( $\mbox{mod } 4$ )  with $r=0$ and $2$.
\end{remark}

Let  $k=n-\alpha=5$. Here we determine the minimizer graph $T^*$ of $\mathbb{G}^5_{n,\alpha}$ in three steps.
{\flushleft\bf Step I.} In the step one, we will find main trees in $\mathcal{T}^*_-(5)$. Now (\ref{eq-t1}) becomes
$$\left\{\begin{array}{ll}d\in \{4,6,8\}\\
h\le \min\{ 4-\frac{d}{2}, \lfloor\frac{ d}{4}\rfloor \}\\
\sum^h_{s=1} |M_s|=4-\frac{d}{2}\\
\sigma^*:M_s\longrightarrow M_{s-1}\\
\ \ \mbox{ for }  1\le s \le h
\end{array}\right.\Rightarrow
\left\{\begin{array}{ll}d=4\\
h= 1\\
\sum^h_{s=1} |M_s|=2\\
\sigma^*:M_s\longrightarrow M_{s-1}\\
\ \ \mbox{ for }  1\le s \le h
\end{array}\right.,\
\left\{\begin{array}{ll}d=6\\
h= 1\\
\sum^h_{s=1} |M_s|=1\\
\sigma^*:M_s\longrightarrow M_{s-1}\\
\ \ \mbox{ for }  1\le s \le h
\end{array}\right.
\mbox{ or }
\left\{\begin{array}{ll}d=8\\
h=0\\
M_s=\emptyset
\end{array}\right.
$$
If the first occurs then $\sigma^*(x_{11}y_{11})=\sigma^*(x_{12}y_{12})=v_1v_2$, which  leads to a main tree $F^5_1=T(4; \{x_{11}y_{11}, x_{12}y_{12}\}; \sigma^*)$ (see Fig.\ref{fig-main-5});
if the second occurs then $\sigma^*(x_{11}y_{11})=v_1v_2$, which  produces a main tree $F^5_2=T(6; \{x_{11}y_{11}\}; \sigma^*)$ (see Fig.\ref{fig-main-5});  if the later occurs we obtain a main tree $F^5_3=T(8; \emptyset; \sigma^*)=P_9$ (see Fig.\ref{fig-main-5}). Thus
 $\mathcal{T}^*_-(5)=\{F^5_1,F^5_2,F^5_3\}$.
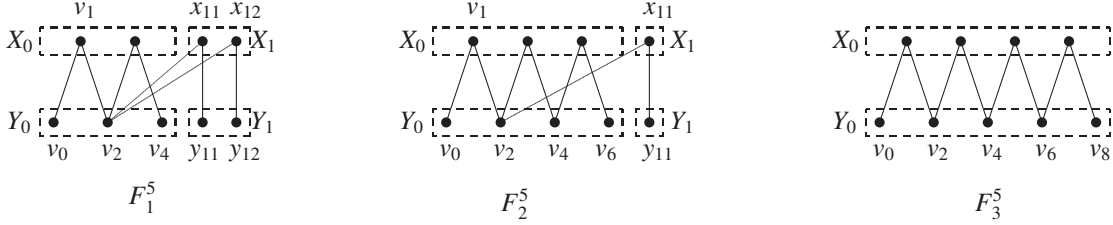
\begin{figure}[h]
  \centering
  \footnotesize

% This is a LaTeX picture output by TeXCAD.
% File name: [Clipboard].
% Version of TeXCAD: 4.3
% Reference / build: 30-Jun-2012 (rev. 105)
% For new versions, check: http://texcad.sf.net/
% Options on the following lines.
%\grade{\on}
%\emlines{\off}
%\epic{\off}
%\beziermacro{\on}
%\reduce{\on}
%\snapping{\on}
%\pvinsert{% Your \input, \def, etc. here}
%\quality{8.000}
%\graddiff{0.005}
%\snapasp{1}
%\zoom{8.0000}
\unitlength 0.9mm % = 2.845pt
\linethickness{0.4pt}
\ifx\plotpoint\undefined\newsavebox{\plotpoint}\fi % GNUPLOT compatibility
\begin{picture}(162,29)(0,0)
\put(68,25){\circle*{1.5}}
\put(68,25){\line(1,-3){4}}
\put(72,13){\circle*{1.5}}
\put(64,13){\circle*{1.5}}
\put(68,25){\line(-1,-3){4}}
\put(76,25){\circle*{1.5}}
\put(76,25){\line(1,-3){4}}
\put(80,13){\circle*{1.5}}
\put(76,25){\line(-1,-3){4}}
\put(84,25){\circle*{1.5}}
\put(84,25){\line(1,-3){4}}
\put(88,13){\circle*{1.5}}
\put(84,25){\line(-1,-3){4}}
\put(94,25){\circle*{1.5}}
\put(94,13){\circle*{1.5}}
\put(94,25){\line(0,-1){12}}
%\emline(94,25)(72,13)
\multiput(94,25)(-.0617977528,-.0337078652){356}{\line(-1,0){.0617977528}}
%\end
\put(57,24){$X_0$}
\put(57,12){$Y_0$}
\put(97,12){$Y_1$}
\put(97,24){$X_1$}
\put(72,0){$F^5_2$}
\put(93,29){$x_{11}$}
\put(93,8){$y_{11}$}
\put(63,8){$v_{0}$}
\put(67,29){$v_{1}$}
\put(71,8){$v_{2}$}
\put(79,8){$v_{4}$}
\put(86,8){$v_{6}$}
\put(10,25){\circle*{1.5}}
\put(10,25){\line(1,-3){4}}
\put(14,13){\circle*{1.5}}
\put(6,13){\circle*{1.5}}
\put(10,25){\line(-1,-3){4}}
\put(18,25){\circle*{1.5}}
\put(18,25){\line(1,-3){4}}
\put(22,13){\circle*{1.5}}
\put(18,25){\line(-1,-3){4}}
\put(28,25){\circle*{1.5}}
\put(28,13){\circle*{1.5}}
\put(28,25){\line(0,-1){12}}
\put(33,25){\circle*{1.5}}
\put(33,13){\circle*{1.5}}
\put(33,25){\line(0,-1){12}}
%\emline(28,25)(14,13)
\multiput(28,25)(-.0393258427,-.0337078652){356}{\line(-1,0){.0393258427}}
%\end
%\emline(33,25)(14,13)
\multiput(33,25)(-.0533707865,-.0337078652){356}{\line(-1,0){.0533707865}}
%\end
\put(-1,24){$X_0$}
\put(-1,12){$Y_0$}
\put(35,24){$X_1$}
\put(35,12){$Y_1$}
\put(17,1){$F^5_1$}
\put(26,29){$x_{11}$}
\put(32,29){$x_{12}$}
\put(26,8){$y_{11}$}
\put(32,8){$y_{12}$}
\put(5,8){$v_{0}$}
\put(9,29){$v_{1}$}
\put(13,8){$v_{2}$}
\put(20,8){$v_{4}$}
\put(132,25){\circle*{1.5}}
\put(132,25){\line(1,-3){4}}
\put(136,13){\circle*{1.5}}
\put(128,13){\circle*{1.5}}
\put(132,25){\line(-1,-3){4}}
\put(140,25){\circle*{1.5}}
\put(140,25){\line(1,-3){4}}
\put(144,13){\circle*{1.5}}
\put(140,25){\line(-1,-3){4}}
\put(148,25){\circle*{1.5}}
\put(148,25){\line(1,-3){4}}
\put(152,13){\circle*{1.5}}
\put(148,25){\line(-1,-3){4}}
\put(156,25){\circle*{1.5}}
\put(156,25){\line(1,-3){4}}
\put(160,13){\circle*{1.5}}
\put(156,25){\line(-1,-3){4}}
\put(127,8){$v_{0}$}
\put(135,8){$v_{2}$}
\put(143,8){$v_{4}$}
\put(151,8){$v_{6}$}
\put(159,8){$v_{8}$}
\put(121,24){$X_0$}
\put(121,12){$Y_0$}
\put(142,0){$F^5_3$}
\put(4,11){\dashbox{1}(20,4)[cc]{}}
\put(4,23){\dashbox{1}(20,4)[cc]{}}
\put(26,23){\dashbox{1}(9,4)[cc]{}}
\put(26,11){\dashbox{1}(9,4)[cc]{}}
\put(62,11){\dashbox{1}(28,4)[cc]{}}
\put(62,23){\dashbox{1}(28,4)[cc]{}}
\put(92,23){\dashbox{1}(4,4)[cc]{}}
\put(92,11){\dashbox{1}(4,4)[cc]{}}
\put(126,11){\dashbox{1}(36,4)[cc]{}}
\put(126,23){\dashbox{1}(36,4)[cc]{}}
\end{picture}

 \caption{\footnotesize The  main trees for $k=5$}\label{fig-main-5}
\end{figure}
{\flushleft\bf Step II. } In the step two, let $n+1\equiv r$ ( $\mbox{mod }5$ ) and $n\ge n_0=3k^2-k-1-(k-1)r=69-4r$, where $0\le r\le 4$, we will  determine the kernel $T_{n_0}$ of minimizer graph $T^*\in \mathbb{G}^5_{n,\alpha}$ according to the different $r$.
Since the kernel $T_{n_0}$ is also
 the minimizer graph in  $\mathbb{G}_{n_0,n_0-5}^5$, by Theorem \ref{thm-minimizer-main}
there exists some main tree $T_-^*=(V_1^*,V_2^*)\in \mathcal{T}^*_-(5)$ such that
 $T_{n_0}=T^*_{-}\circ l_{V_2^*}$ with leaf sequence $l_{V_2^*}=(l(u)\mid u\in V_2^*)$ satisfying $|L(T_{n_0})|=\sum_{u\in V_2^*}l(u)=n_0-9=60-4r$. It remains to find $l(u)$ of $T_{n_0}$ for $u\in V_2^*$.
 Let $\bar{l}=\lfloor\frac{|L(T_{n_0})|}{|V_2^*|}\rfloor=\lfloor\frac{60-4r}{5}\rfloor$.

First we suppose that $r=0$. Then $n_0=69$ and $\bar{l}=12$.
 From (\ref{ll-eq-2}),   $(l(u)\mid u\in V_2^*)$  satisfies
\begin{equation}\label{eq-sum-1}
\left\{\begin{array}{ll}\sum_{u\in V_2^*}l(u)=60\\
8-d_{T^*_-}(u)\le l(u)\le 15-d_{T^*_{-}}(u) \mbox{ for any $u\in V_2^*$}\end{array}\right.\end{equation}
where $d_{T^*_-}(u)$ can be determined by a main tree selected from the set $\mathcal{T}^*_-(5)=\{F^5_1,F^5_2,F^5_3\}$.

If $T^*_-=F^5_1=(V^*_1,V^*_2)$, from Fig.\ref{fig-main-5} we see that $V^*_2=\{v_0,v_2,v_4, y_{11}, y_{12}\}$ and $d_{F^5_1}(v_{0})=d_{F^5_1}(v_{4})=d_{F^5_1}(y_{11})=d_{F^5_1}(y_{12})=1$
and $d_{F^5_1}(v_{2})=4$.
Bringing these values into (\ref{eq-sum-1}), we get  $38$ solutions for $l_{V_2^*}=(l(u)\mid u\in V_2^*)$ which is collected in the set $\mathcal{L}_{F^5_1}$ and listed in Tab.\ref{tab-solution}.
By Theorem \ref{thm-minimizer-main}, $\mathbb{F}^5_1=\{F^5_1\circ l_{V_2^*}\mid l_{V_2^*}\in \mathcal{L}_{F^5_1}\} $ contains $38$ possible  minimizer graphs in $\mathbb{G}_{n_0,n_0-5}^5$.
Comparing their spectral radii, we obtain   $T^5_1=F^5_1\circ (13,8,13,13,13)$  with minimum spectral radius $\sqrt{13+\sqrt{5}}$ among $\mathbb{F}^5_1$  (see the first line in Tab.\ref{table-1}).

\begin{table}[H]
\footnotesize
\centering
\begin{tabular*}{14.6cm}{p{72pt}p{72pt}p{72pt}p{72pt}p{72pt}}
\hline
\multicolumn{5}{c}{$(\ l(v_0),\ l(v_2), \ l(v_4),\  l(y_{11}), \ l(y_{12})\ )$}\\ \hline
$(14,	4,14,	14,	14)$& $(13,5,	14,	14,	14)$&$(12,6,	14,	14,	14)$& $(13,6,	13,	14,	14)$& $(11,7,	14,	14,	14)$\\
 $(12,7,	13,	14,	14)$& $(13,7,	13,	13,	14)$& $(10,8,	14,	14,	14)$&  $(11, 8,	13,	14,	14)$& $(12,8,	12,	14,	14)$\\ $(12, 8,	13,	13,	14)$& \bm{$(13, 8,	13,	13,	13)$}&  $(9, 9,	14,	14,	14)$& $(10,9,	13,	14,	14)$&  $(11,9,	12,	14,	14)$\\ $(11,9,	13,	13,	14)$& $(12,9,	12,	13,	14)$& $(12,9,	13,	13,	13)$& $(8,10,	14,	14,	14)$ & $(9,10,	13,	14,	14)$\\ $(10,10,	12,	14,	14)$& $(10,10,	13,	13,	14)$ &$(11,10,	11,	14,	14)$&  $(11,10,	12,	13,	14)$& $(11,10,	13,	13,	13)$\\ $(12,10,	12,	12,	14)$& $(12,10,	12,	13,	13)$ & $(7,11,	14,	14,	14)$& $(8,11,	13,	14,	14)$&  $(9,11,	12,	14,	14)$\\ $(9,11,	13,	13,	14)$&  $(10,11,	11,	14,	14)$&  $(10,11,	12,	13,	14)$& $(10,11,	13,	13,	13)$ & $(11,11,	11,	13,	14)$\\ $(11,11,	12,	12,	14)$&  $(11,11,	12,	13,	13)$&  $(12,11,	12,	12,	13)$&&\\ \hline
\end{tabular*}
  \caption{\small The set $\mathcal{L}_{F^5_1}$}\label{tab-solution}
\end{table}

If $T^*_-=F^5_2=(V^*_1,V^*_2)$, from Fig.\ref{fig-main-5} we see that $V^*_2=\{v_0,v_2,v_4, v_6, y_{11}\}$ and  $d_{F^5_2}(v_{0})=d_{F^5_2}(v_{6})=d_{F^5_2}(y_{11})=1$,
$d_{F^5_2}(v_{2})=3$ and $d_{F^5_2}(v_{4})=2$. As similar as above,
 we get $200$ solutions of (\ref{eq-sum-1})  which is collected in $\mathcal{L}_{F^5_2}$.  Thus $\mathbb{F}^5_2=\{F^5_2\circ l_{V_2^*}\mid l_{V_2^*}\in \mathcal{L}_{F^5_2}\} $ contains $200$
possible  minimizer graphs, by comparing spectral radii  we get  $T^5_2=F^5_2\circ (13,10,11,13,13)$ with minimum spectral radius $3.9068$ among $\mathbb{F}^5_2$ (see the second line in Tab.\ref{table-1}).

If  $T^*_-=F^5_3=(V^*_1,V^*_2)$, from Fig.\ref{fig-main-5} we have $V^*_2=\{v_0,v_2,v_4, v_6, v_8\}$ and  $d_{F^5_3}(v_{0})=d_{F^5_3}(v_{8})=1$,
$d_{F^5_3}(v_{2})=d_{F^5_3}(v_{4})=d_{F^5_3}(v_{6})=2$.
Similarly, we can put 170 solutions of (\ref{eq-sum-1})  in $\mathcal{L}_{F^5_2}$ and find that  $T^5_3=F^5_3\circ (13,11,12,11,13)$ is the minimizer graph with spectral radius $\sqrt{\frac{27+\sqrt{13}}{2}}$ among $\mathbb{F}^5_3=\{F^5_3\circ l_{V_2^*}\mid l_{V_2^*}\in \mathcal{L}_{F^5_3}\}$ (see the third line in Tab.\ref{table-1}).

Finally, by comparing the spectral radii of $T^5_1$, $T^5_2$ and $T^5_3$,  we get
 $T_{n_0}=T^5_1=F^5_1\circ (13,8,13,13,13)$ is the minimizer graph in  $\mathbb{G}_{69,64}^5$ with respect to $r=0$.

Follow the same procedure as  $r=0$, we can obtain the minimizer graph $T_{n_0}$ in  $\mathbb{G}_{n_0,n_0-5}^5$ with $n_0=69-4r$ for $r=1,2,3$ and $4$, which are all listed in Tab.\ref{table-1}, respectively.
That is $T_{n_0}=F^5_1\circ(13,8,13,13,13)$, $F^5_3\circ(12,11,10,11,12)$ and $F^5_3\circ(12,9,10,9,12)$ if $r=0,1,2$ and $3$, respectively, and $T_{n_0}=F^5_1\circ(10,4,10,10,10)$, $F^5_2\circ(10,6,8,10,10)$ or $F^5_3\circ(10,8,8,8,10)$ if $r=4$.
\begin{table}[H]
\footnotesize
\caption{\small The  kernel  $T_{n_0}$ for $k=5$}
\centering
\renewcommand{\arraystretch}{1.25}
\begin{tabular*}{15.00cm}{p{2pt}|p{3pt}|p{6pt}|p{135pt}|p{13pt}|p{13pt}|p{120pt}|p{35pt}}
\hline
$k$&  $r$& $n_0$& the condition (\ref{eq-sum-1})
  %satisfied by $(l(u)\mid u\in V_2^*)$
  & $T^*_-$ &  \#
  & graph $T^5_i$   & $\rho(T^5_i)$ \\ \hline
\multirow{15}*{5}&\multirow{3}*{0}&\multirow{3}*{$69$}&
 \multirow{3}*{$\left\{\begin{array}{ll}\sum_{u\in V_2^*}l(u)=60\vspace{0.08cm}\\8-d_{T^*_-}(u)\le l(u)\le 15-d_{T^*_{-}}(u)
 \end{array}\right.$}&
 $F^5_1$&$38$&
\bm{$F^5_1\circ(13,8,13,13,13)=T_{n_0}$}& \bm{$\!\!\sqrt{13+\!\sqrt{5}}$}\vspace{0.15cm}\\ \cline{5-8}
&&&&$F^5_2$&$200$&$F^5_2\circ(13,10,11,13,13)$& $3.9068$ \vspace{0.15cm}\\ \cline{5-8}
&&&&$F^5_3$&$170$&$F^5_3\circ(13,11,12,11,13)$& $\sqrt{\frac{27+\sqrt{13}}{2}}$\vspace{0.08cm} \\
\cline{2-8}
& \multirow{3}*{1}&\multirow{3}*{$65$}&
 \multirow{3}*{$\left\{\begin{array}{ll}\sum_{u\in V_2^*}l(u)=56\vspace{0.08cm}\\8-d_{T^*_-}(u)\le l(u)\le 14-d_{T^*_{-}}(u)
 \end{array}\right.$}&
 $F^5_1$&$27$&$F^5_1\circ(12,8,12,12,12)$& $\sqrt{\frac{25+\sqrt{17}}{2}}$\vspace{0.08cm}\\ \cline{5-8}
&&&&$F^5_2$&$130$&$F^5_2\circ(12,9,11,12,12)$& $3.8090$ \vspace{0.15cm}\\ \cline{5-8}
&&&&$F^5_3$&$110$&\bm{$F^5_3\circ(12,11,10,11,12)=T_{n_0}$}& \bm{$3.8054$}\vspace{0.15cm} \\
\cline{2-8}
&\multirow{3}*{2}&\multirow{3}*{$61$}&
 \multirow{3}*{$\left\{\begin{array}{ll}\sum_{u\in V_2^*}l(u)=52\vspace{0.08cm}\\8-d_{T^*_-}(u)\le l(u)\le 13-d_{T^*_{-}}(u)
 \end{array}\right.$}&
 $F^5_1$&$18$&$F^5_1\circ(12,4,12,12,12)$& $\sqrt{\frac{21+\sqrt{41}}{2}}$ \\ \cline{5-8}
&&&&$F^5_2$&$80$&$F^5_2\circ(12,7,9,12,12)$& $3.7003$ \vspace{0.15cm}\\ \cline{5-8}
&&&&$F^5_3$&$66$&\bm{$F^5_3\circ(12,9,10,9,12)=T_{n_0}$}& \bm{$3.6980$}\vspace{0.15cm} \\
\cline{2-8}
&\multirow{3}*{3}&\multirow{3}*{$57$}&
 \multirow{3}*{$\left\{\begin{array}{ll}\sum_{u\in V_2^*}l(u)=48\vspace{0.08cm}\\8-d_{T^*_-}(u)\le l(u)\le 12-d_{T^*_{-}}(u)
 \end{array}\right.$}&$F^5_1$&$12$&\bm{$F^5_1\circ(11,4,11,11,11)=T_{n_0}$}& \bm{$\!\!\sqrt{10+\!\sqrt{8}}$}\vspace{0.1cm}\\ \cline{5-8}
&&&&$F^5_2$&$46$&$F^5_2\circ(11,6,9,11,11)$& $3.5845$ \vspace{0.15cm}\\ \cline{5-8}
&&&&$F^5_3$&$38$&$F^5_3\circ(11,9,8,9,11)$& $3.5820$\vspace{0.15cm} \\
\cline{2-8}
&\multirow{3}*{4}&\multirow{3}*{$53$}&
 \multirow{3}*{$\left\{\begin{array}{ll}\sum_{u\in V_2^*}l(u)=44\vspace{0.08cm}\\8-d_{T^*_-}(u)\le l(u)\le 11-d_{T^*_{-}}(u)
 \end{array}\right.$}&$F^5_1$&$7$&\bm{$F^5_1\circ(10,4,10,10,10)=T_{n_0}$}& \bm{$2\sqrt{3}$} \vspace{0.15cm}\\ \cline{5-8}
&&&&$F^5_2$&$24$&\bm{$F^5_2\circ(10,6,8,10,10)=T_{n_0}$}& \bm{$2\sqrt{3}$} \vspace{0.15cm}\\ \cline{5-8}
&&&&$F^5_3$&$19$&\bm{$F^5_3\circ(10,8,8,8,10)=T_{n_0}$}& \bm{$2\sqrt{3}$}\vspace{0.15cm} \\
\hline
\multicolumn{8}{l}{\begin{tabular}{@{}l@{}}\# indicates the number of the solutions of $l_{V_2^*}=(l(u)\mid u\in V_2^*)$  satisfying the condition (\ref{eq-sum-1}) in the fourth
\\ column for given $T^*_-\in \mathcal{T}^*_-(5)$.\end{tabular} }
\end{tabular*}\label{table-1}
\end{table}

{\flushleft\bf Step III. } In the step three, we will determine the minimizer graph $T^*$ in $\mathbb{G}^5_{n,\alpha}$ for any
$n\ge n_0=69-4r$ with $0\le r\le 4$.
By Theorem \ref{thm-minimizer-spectral-radius} we have $T^*=T_{n_0}\circ \ell_{n,5} \mathbf{1}_{V_2^*}$ and $\rho(T^*)=\sqrt{\rho^2(T_{n_0})+\ell_{n,5}}$, where $\ell_{n,5}=\frac{n-n_0}{5}=\frac{n-(69-4r)}{5}$.
For $r=0$, we know  $T_{n_0}=F^5_1\circ(13,8,13,13,13)$ and
$\rho^2(T_{n_0})=13+\sqrt{5}$.
Note that $\ell_{n,5}=\frac{n-69}{5}$,
 we have
$$\begin{array}{ll}T^*&=T_{n_0}\circ \ell_{n,5} \mathbf{1}_{V_2^*}=T_{n_0}\circ \frac{n-69}{5} \mathbf{1}_{V_2^*}\vspace{0.2cm}\\
&=F^5_1\circ(13+\frac{n-69}{5} ,8+\frac{n-69}{5} ,13+\frac{n-69}{5} ,13+\frac{n-69}{5} ,13+\frac{n-69}{5} )\vspace{0.2cm}\\&=
F^5_1 \circ(\frac{n-4}{5},\frac{n-29}{5},\frac{n-4}{5},\frac{n-4}{5},\frac{n-4}{5})=T^*_{n,0}\ \ \mbox{(see Fig.\ref{111})}
\end{array}$$
 is the  minimizer graph in $\mathbb{G}_{n,n-5}^5$ with $\rho(T^*)=\sqrt{13+\sqrt{5}+\ell_{n,5}}=\sqrt{\frac{n-4}{5}+\sqrt{5}}$.
 %by Lemma \ref{lem-bipartite}.

As in the case of $r=0$, we can obtain the minimizer graph $T^*=T^*_{n,r}$ along with the spectral radius for $r=1,2,3$ and $4$, respectively, which are shown in
Fig.\ref{111} and omit the specific calculations. Finally we summarize these results in the following Theorem \ref{thm-n-5}.

\begin{thm}\label{thm-n-5}
Let $T^*$ be the  minimizer  graph  in $\mathbb{G}^5_{n,\alpha}$ and let $n+1\equiv r$ ( $\mbox{mod }5$ ), where $0\le r \le 4$. For $n\ge 69-4r$,  we have
$$T^*=\left\{
\begin{array}{ll}
T^*_{n,0}=F^5_1 \circ(\frac{n-4}{5},\frac{n-29}{5},\frac{n-4}{5},\frac{n-4}{5},\frac{n-4}{5})& \mbox {if $r=0$,}\vspace{0.15cm}\\
T^*_{n,1}=F^5_3 \circ (\frac{n-5}{5},\frac{n-10}{5},\frac{n-15}{5},\frac{n-10}{5},\frac{n-5}{5})& \mbox {if $r=1$,}\vspace{0.15cm}\\
T^*_{n,2}=F^5_3  \circ (\frac{n-1}{5},\frac{n-16}{5},\frac{n-11}{5},\frac{n-16}{5},\frac{n-1}{5})& \mbox {if $r=2$,}\vspace{0.15cm}\\
T^*_{n,3}=F^5_1 \circ (\frac{n-2}{5},\frac{n-37}{5},\frac{n-2}{5},\frac{n-2}{5},\frac{n-2}{5})& \mbox {if $r=3$,}\vspace{0.15cm}\\
T^*_{n,4}=\left\{\begin{array}{ll}F^5_1\circ(\frac{n-3}{5},\frac{n-33}{5},\frac{n-3}{5},\frac{n-3}{5},\frac{n-3}{5})
\vspace{0.15cm}\\
F^5_2 \circ (\frac{n-3}{5},\frac{n-23}{5},\frac{n-13}{5},\frac{n-3}{5},\frac{n-3}{5})\vspace{0.15cm}\\
F^5_3 \circ (\frac{n-3}{5},\frac{n-13}{5},\frac{n-13}{5},\frac{n-13}{5},\frac{n-3}{5})
\end{array}\right. & \mbox {if $r= 4$,}\\
\end{array}\right.$$
where minimizer graphs $T^*_{n,0}$, $T^*_{n,1}$,..., $T^*_{n,4}$ are described in Fig.\ref{111}.
Moreover, the spectral radius of $T^*$ is $\rho(T^*_{n,0})=\sqrt{\frac{n-4}{5}+\sqrt{5}}$, $\rho(T^*_{n,1})=\sqrt{\frac{n-5}{5}+2.4812}$, $\rho(T^*_{n,2})=\sqrt{\frac{n-6}{5}+2.6751}$, $\rho(T^*_{n,3})=\sqrt{\frac{n-7}{5}+\sqrt{8}}$ and  $\rho(T^*_{n,4})=\sqrt{\frac{n+7}{5}}$, respectively.
\end{thm}

\begin{figure}[h]
  \centering
  \footnotesize
  % This is a LaTeX picture output by TeXCAD.
% File name: [Clipboard].
% Version of TeXCAD: 4.3
% Reference / build: 30-Jun-2012 (rev. 105)
% For new versions, check: http://texcad.sf.net/
% Options on the following lines.
%\grade{\on}
%\emlines{\off}
%\epic{\off}
%\beziermacro{\on}
%\reduce{\on}
%\snapping{\on}
%\pvinsert{% Your \input, \def, etc. here}
%\quality{8.000}
%\graddiff{0.005}
%\snapasp{1}
%\zoom{6.7272}
% This is a LaTeX picture output by TeXCAD.
% File name: [Clipboard].
% Version of TeXCAD: 4.3
% Reference / build: 30-Jun-2012 (rev. 105)
% For new versions, check: http://texcad.sf.net/
% Options on the following lines.
%\grade{\on}
%\emlines{\off}
%\epic{\off}
%\beziermacro{\on}
%\reduce{\on}
%\snapping{\on}
%\pvinsert{% Your \input, \def, etc. here}
%\quality{8.000}
%\graddiff{0.005}
%\snapasp{1}
%\zoom{6.7272}
\unitlength 0.88mm % = 2.845pt
\linethickness{0.4pt}
\ifx\plotpoint\undefined\newsavebox{\plotpoint}\fi % GNUPLOT compatibility
\begin{picture}(156.75,77.75)(0,0)
\put(14,48){\bm{$T^*_{n,0}=F^5_1 \circ(\frac{n-4}{5},\frac{n-29}{5},\frac{n-4}{5},\frac{n-4}{5},\frac{n-4}{5})$}}
\put(12,42){( $T^*_{n,3}=F^5_1 \circ(\frac{n-2}{5},\frac{n-37}{5},\frac{n-2}{5},\frac{n-2}{5},\frac{n-2}{5})$ )}
%\emline(21,77)(14,67)
\multiput(21,77)(-.033653846,-.048076923){208}{\line(0,-1){.048076923}}
%\end
%\emline(21,77)(28,67)
\multiput(21,77)(.033653846,-.048076923){208}{\line(0,-1){.048076923}}
%\end
\put(21,77){\circle*{1.5}}
\put(28,67){\circle*{1.5}}
\put(14,67){\circle*{1.5}}
\put(14,67){\line(-3,-5){3}}
\put(14,67){\line(3,-5){3}}
\put(11,62){\circle*{1.5}}
\put(17,62){\circle*{1.5}}
\put(13,61){$\cdots$}
\put(9,61){$\underbrace{}_{\bm{\frac{n-4}{5}}(\frac{n-2}{5})}$}
\put(15,66){$v_{0}$}
\put(28,67){\line(-3,-5){3}}
\put(28,67){\line(3,-5){3}}
\put(25,62){\circle*{1.5}}
\put(31,62){\circle*{1.5}}
\put(27,61){$\cdots$}
\put(22,61){$\underbrace{}_{\bm{\frac{n-29}{5}}(\frac{n-37}{5})}$}
\put(29,66){$v_{2}$}
%\emline(35,77)(28,67)
\multiput(35,77)(-.033653846,-.048076923){208}{\line(0,-1){.048076923}}
%\end
%\emline(35,77)(42,67)
\multiput(35,77)(.033653846,-.048076923){208}{\line(0,-1){.048076923}}
%\end
\put(35,77){\circle*{1.5}}
\put(42,67){\circle*{1.5}}
\put(42,67){\line(-3,-5){3}}
\put(42,67){\line(3,-5){3}}
\put(39,62){\circle*{1.5}}
\put(45,62){\circle*{1.5}}
\put(41,61){$\cdots$}
\put(37,61){$\underbrace{}_{\bm{\frac{n-4}{5}}(\frac{n-2}{5})}$}
\put(43,66){$v_{4}$}
\put(56,67){\circle*{1.5}}
\put(56,66){\line(0,1){11}}
\put(56,77){\circle*{1.5}}
\put(70,67){\circle*{1.5}}
\put(70,66){\line(0,1){11}}
\put(70,77){\circle*{1.5}}
%\emline(56,77)(28,67)
\multiput(56,77)(-.0942760943,-.0336700337){297}{\line(-1,0){.0942760943}}
%\end
%\emline(70,77)(28,67)
\multiput(70,77)(-.1414141414,-.0336700337){297}{\line(-1,0){.1414141414}}
%\end
\put(56,67){\line(-3,-5){3}}
\put(56,67){\line(3,-5){3}}
\put(53,62){\circle*{1.5}}
\put(59,62){\circle*{1.5}}
\put(55,61){$\cdots$}
\put(51,61){$\underbrace{}_{\bm{\frac{n-4}{5}}(\frac{n-2}{5})}$}
\put(57,66){$y_{11}$}
\put(70,67){\line(-3,-5){3}}
\put(70,67){\line(3,-5){3}}
\put(67,62){\circle*{1.5}}
\put(73,62){\circle*{1.5}}
\put(69,61){$\cdots$}
\put(65,61){$\underbrace{}_{\bm{\frac{n-4}{5}}(\frac{n-2}{5})}$}
\put(71,66){$y_{12}$}
\put(95,48){\bm{$T^*_{n,1}=F^5_3 \circ (\frac{n-5}{5},\frac{n-10}{5},\frac{n-15}{5},\frac{n-10}{5},\frac{n-5}{5})$}}
\put(93,42){( $T^*_{n,2}=F^5_3 \circ (\frac{n-1}{5},\frac{n-16}{5},\frac{n-11}{5},\frac{n-16}{5},\frac{n-1}{5})$ )}
%\emline(104,77)(97,67)
\multiput(104,77)(-.033653846,-.048076923){208}{\line(0,-1){.048076923}}
%\end
%\emline(104,77)(111,67)
\multiput(104,77)(.033653846,-.048076923){208}{\line(0,-1){.048076923}}
%\end
\put(104,77){\circle*{1.5}}
\put(111,67){\circle*{1.5}}
\put(97,67){\circle*{1.5}}
\put(97,67){\line(-3,-5){3}}
\put(97,67){\line(3,-5){3}}
\put(94,62){\circle*{1.5}}
\put(100,62){\circle*{1.5}}
\put(96,61){$\cdots$}
\put(92,61){$\underbrace{}_{\bm{\frac{n-5}{5}}(\frac{n-1}{5})}$}
\put(98,66){$v_{0}$}
\put(111,67){\line(-3,-5){3}}
\put(111,67){\line(3,-5){3}}
\put(108,62){\circle*{1.5}}
\put(114,62){\circle*{1.5}}
\put(110,61){$\cdots$}
\put(105,61){$\underbrace{}_{\bm{\frac{n-10}{5}}(\frac{n-16}{5})}$}
\put(112,66){$v_{2}$}
%\emline(118,77)(111,67)
\multiput(118,77)(-.033653846,-.048076923){208}{\line(0,-1){.048076923}}
%\end
%\emline(118,77)(125,67)
\multiput(118,77)(.033653846,-.048076923){208}{\line(0,-1){.048076923}}
%\end
\put(118,77){\circle*{1.5}}
\put(125,67){\circle*{1.5}}
\put(125,67){\line(-3,-5){3}}
\put(125,67){\line(3,-5){3}}
\put(122,62){\circle*{1.5}}
\put(128,62){\circle*{1.5}}
\put(124,61){$\cdots$}
\put(119,61){$\underbrace{}_{\bm{\frac{n-15}{5}}(\frac{n-11}{5})}$}
\put(126,66){$v_{4}$}
%\emline(132,77)(125,67)
\multiput(132,77)(-.033653846,-.048076923){208}{\line(0,-1){.048076923}}
%\end
%\emline(132,77)(139,67)
\multiput(132,77)(.033653846,-.048076923){208}{\line(0,-1){.048076923}}
%\end
\put(132,77){\circle*{1.5}}
\put(139,67){\circle*{1.5}}
\put(139,67){\line(-3,-5){3}}
\put(139,67){\line(3,-5){3}}
\put(136,62){\circle*{1.5}}
\put(142,62){\circle*{1.5}}
\put(138,61){$\cdots$}
\put(133,61){$\underbrace{}_{\bm{\frac{n-10}{5}}(\frac{n-16}{5})}$}
\put(140,66){$v_{6}$}
%\emline(146,77)(139,67)
\multiput(146,77)(-.033653846,-.048076923){208}{\line(0,-1){.048076923}}
%\end
%\emline(146,77)(153,67)
\multiput(146,77)(.033653846,-.048076923){208}{\line(0,-1){.048076923}}
%\end
\put(146,77){\circle*{1.5}}
\put(153,67){\circle*{1.5}}
\put(153,67){\line(-3,-5){3}}
\put(153,67){\line(3,-5){3}}
\put(150,62){\circle*{1.5}}
\put(156,62){\circle*{1.5}}
\put(152,61){$\cdots$}
\put(148,61){$\underbrace{}_{\bm{\frac{n-5}{5}}(\frac{n-1}{5})}$}
\put(154,66){$v_{8}$}
\put(151,0){\line(0,1){0}}
\put(9,34){\circle*{1.5}}
%\emline(9,34)(14,25)
\multiput(9,34)(.033557047,-.060402685){149}{\line(0,-1){.060402685}}
%\end
%\emline(9,34)(4,25)
\multiput(9,34)(-.033557047,-.060402685){149}{\line(0,-1){.060402685}}
%\end
\put(19,34){\circle*{1.5}}
%\emline(19,34)(24,25)
\multiput(19,34)(.033557047,-.060402685){149}{\line(0,-1){.060402685}}
%\end
%\emline(19,34)(14,25)
\multiput(19,34)(-.033557047,-.060402685){149}{\line(0,-1){.060402685}}
%\end
\put(4,25){\circle*{1.5}}
\put(14,25){\circle*{1.5}}
\put(24,25){\circle*{1.5}}
\put(4,25){\line(-3,-5){3}}
\put(4,25){\line(3,-5){3}}
\put(1,20){\circle*{1.5}}
\put(7,20){\circle*{1.5}}
\put(14,25){\line(-3,-5){3}}
\put(14,25){\line(3,-5){3}}
\put(11,20){\circle*{1.5}}
\put(17,20){\circle*{1.5}}
\put(24,25){\line(-3,-5){3}}
\put(24,25){\line(3,-5){3}}
\put(21,20){\circle*{1.5}}
\put(27,20){\circle*{1.5}}
\put(33,25){\line(-3,-5){3}}
\put(33,25){\line(3,-5){3}}
\put(30,20){\circle*{1.5}}
\put(36,20){\circle*{1.5}}
\put(42,25){\line(-3,-5){3}}
\put(42,25){\line(3,-5){3}}
\put(39,20){\circle*{1.5}}
\put(45,20){\circle*{1.5}}
\put(33,34){\circle*{1.5}}
\put(33,34){\line(0,-1){9}}
\put(33,25){\circle*{1.5}}
\put(42,34){\circle*{1.5}}
\put(42,34){\line(0,-1){9}}
\put(42,25){\circle*{1.5}}
\put(3,19){$\cdots$}
\put(0,19){$\underbrace{}_{\frac{n-3}{5}}$}
\put(10,19){$\underbrace{}_{\frac{n-33}{5}}$}
\put(20,19){$\underbrace{}_{\frac{n-3}{5}}$}
\put(29,19){$\underbrace{}_{\frac{n-3}{5}}$}
\put(38,19){$\underbrace{}_{\frac{n-3}{5}}$}
\put(13,19){$\cdots$}
\put(23,19){$\cdots$}
\put(32,19){$\cdots$}
\put(41,19){$\cdots$}
%\emline(33,34)(14,25)
\multiput(33,34)(-.0711610487,-.0337078652){267}{\line(-1,0){.0711610487}}
%\end
%\emline(42,34)(14,25)
\multiput(42,34)(-.1048689139,-.0337078652){267}{\line(-1,0){.1048689139}}
%\end
\put(5,24){$v_{0}$}
\put(15,24){$v_{2}$}
\put(25,24){$v_{4}$}
\put(34,24){$y_{11}$}
\put(43,24){$y_{12}$}
\put(-3,6){$T^*_{n,4}\!\!=\!\!F^5_1\!\circ\!(\frac{n-3}{5},\!\frac{n-33}{5},\!\frac{n-3}{5},\!\frac{n-3}{5},\!\frac{n-3}{5})$}
\put(120,34){\circle*{1.5}}
%\emline(120,34)(125,25)
\multiput(120,34)(.033557047,-.060402685){149}{\line(0,-1){.060402685}}
%\end
%\emline(120,34)(115,25)
\multiput(120,34)(-.033557047,-.060402685){149}{\line(0,-1){.060402685}}
%\end
\put(130,34){\circle*{1.5}}
%\emline(130,34)(135,25)
\multiput(130,34)(.033557047,-.060402685){149}{\line(0,-1){.060402685}}
%\end
%\emline(130,34)(125,25)
\multiput(130,34)(-.033557047,-.060402685){149}{\line(0,-1){.060402685}}
%\end
\put(115,25){\circle*{1.5}}
\put(125,25){\circle*{1.5}}
\put(135,25){\circle*{1.5}}
\put(115,25){\line(-3,-5){3}}
\put(115,25){\line(3,-5){3}}
\put(112,20){\circle*{1.5}}
\put(118,20){\circle*{1.5}}
\put(125,25){\line(-3,-5){3}}
\put(125,25){\line(3,-5){3}}
\put(122,20){\circle*{1.5}}
\put(128,20){\circle*{1.5}}
\put(135,25){\line(-3,-5){3}}
\put(135,25){\line(3,-5){3}}
\put(132,20){\circle*{1.5}}
\put(138,20){\circle*{1.5}}
\put(144,25){\line(-3,-5){3}}
\put(144,25){\line(3,-5){3}}
\put(141,20){\circle*{1.5}}
\put(147,20){\circle*{1.5}}
\put(153,25){\line(-3,-5){3}}
\put(153,25){\line(3,-5){3}}
\put(150,20){\circle*{1.5}}
\put(156,20){\circle*{1.5}}
\put(144,25){\circle*{1.5}}
\put(153,25){\circle*{1.5}}
\put(114,19){$\cdots$}
\put(111,19){$\underbrace{}_{\frac{n-3}{5}}$}
\put(121,19){$\underbrace{}_{\frac{n-13}{5}}$}
\put(131,19){$\underbrace{}_{\frac{n-13}{5}}$}
\put(140,19){$\underbrace{}_{\frac{n-13}{5}}$}
\put(149,19){$\underbrace{}_{\frac{n-3}{5}}$}
\put(124,19){$\cdots$}
\put(134,19){$\cdots$}
\put(143,19){$\cdots$}
\put(152,19){$\cdots$}
\put(116,24){$v_{0}$}
\put(126,24){$v_{2}$}
\put(136,24){$v_{4}$}
\put(145,24){$v_{6}$}
\put(154,24){$v_{8}$}
%\emline(140,34)(135,25)
\multiput(140,34)(-.033557047,-.060402685){149}{\line(0,-1){.060402685}}
%\end
%\emline(140,34)(144,25)
\multiput(140,34)(.033613445,-.075630252){119}{\line(0,-1){.075630252}}
%\end
\put(140,34){\circle*{1.5}}
%\emline(149,34)(144,25)
\multiput(149,34)(-.033557047,-.060402685){149}{\line(0,-1){.060402685}}
%\end
%\emline(149,34)(153,25)
\multiput(149,34)(.033613445,-.075630252){119}{\line(0,-1){.075630252}}
%\end
\put(149,34){\circle*{1.5}}
\put(109,6){$T^*_{n,4}\!\!=\!\!F^5_3 \!\circ\! (\frac{n-3}{5},\!\frac{n-13}{5},\!\frac{n-13}{5},\!\frac{n-13}{5},\!\frac{n-3}{5})$}
\put(64,34){\circle*{1.5}}
%\emline(64,34)(69,25)
\multiput(64,34)(.033557047,-.060402685){149}{\line(0,-1){.060402685}}
%\end
%\emline(64,34)(59,25)
\multiput(64,34)(-.033557047,-.060402685){149}{\line(0,-1){.060402685}}
%\end
\put(74,34){\circle*{1.5}}
%\emline(74,34)(79,25)
\multiput(74,34)(.033557047,-.060402685){149}{\line(0,-1){.060402685}}
%\end
%\emline(74,34)(69,25)
\multiput(74,34)(-.033557047,-.060402685){149}{\line(0,-1){.060402685}}
%\end
\put(59,25){\circle*{1.5}}
\put(69,25){\circle*{1.5}}
\put(79,25){\circle*{1.5}}
\put(59,25){\line(-3,-5){3}}
\put(59,25){\line(3,-5){3}}
\put(56,20){\circle*{1.5}}
\put(62,20){\circle*{1.5}}
\put(69,25){\line(-3,-5){3}}
\put(69,25){\line(3,-5){3}}
\put(66,20){\circle*{1.5}}
\put(72,20){\circle*{1.5}}
\put(79,25){\line(-3,-5){3}}
\put(79,25){\line(3,-5){3}}
\put(76,20){\circle*{1.5}}
\put(82,20){\circle*{1.5}}
\put(88,25){\line(-3,-5){3}}
\put(88,25){\line(3,-5){3}}
\put(85,20){\circle*{1.5}}
\put(91,20){\circle*{1.5}}
\put(97,25){\line(-3,-5){3}}
\put(97,25){\line(3,-5){3}}
\put(94,20){\circle*{1.5}}
\put(100,20){\circle*{1.5}}
\put(88,25){\circle*{1.5}}
\put(97,34){\circle*{1.5}}
\put(97,34){\line(0,-1){9}}
\put(97,25){\circle*{1.5}}
\put(58,19){$\cdots$}
\put(55,19){$\underbrace{}_{\frac{n-3}{5}}$}
\put(65,19){$\underbrace{}_{\frac{n-23}{5}}$}
\put(75,19){$\underbrace{}_{\frac{n-13}{5}}$}
\put(84,19){$\underbrace{}_{\frac{n-3}{5}}$}
\put(93,19){$\underbrace{}_{\frac{n-3}{5}}$}
\put(68,19){$\cdots$}
\put(78,19){$\cdots$}
\put(87,19){$\cdots$}
\put(96,19){$\cdots$}
%\emline(97,34)(69,25)
\multiput(97,34)(-.1048689139,-.0337078652){267}{\line(-1,0){.1048689139}}
%\end
\put(60,24){$v_{0}$}
\put(70,24){$v_{2}$}
\put(80,24){$v_{4}$}
\put(89,24){$v_{6}$}
\put(98,24){$y_{11}$}
%\emline(84,34)(79,25)
\multiput(84,34)(-.033557047,-.060402685){149}{\line(0,-1){.060402685}}
%\end
%\emline(84,34)(88,25)
\multiput(84,34)(.033613445,-.075630252){119}{\line(0,-1){.075630252}}
%\end
\put(84,34){\circle*{1.5}}
\put(52,6){$T^*_{n,4}\!\!=\!\!F^5_2 \!\circ \! (\frac{n-3}{5},\!\frac{n-23}{5},\!\frac{n-13}{5},\!\frac{n-3}{5},\!\frac{n-3}{5})$}
\end{picture}

 \caption{\footnotesize The minimizer graphs for $k=5$}\label{111}
\end{figure}
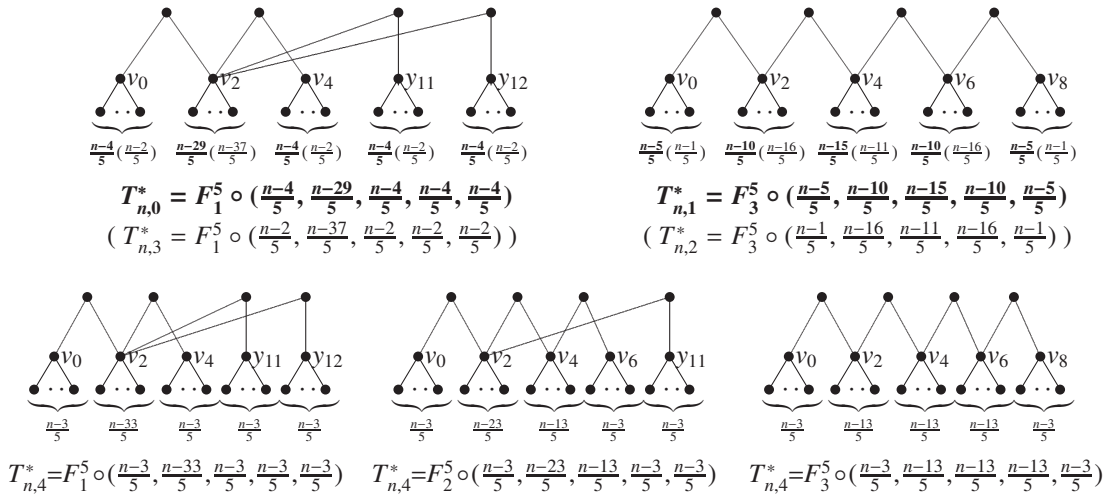

In the rest of this section, let $k=n-\alpha=6$, we briefly
repeat the steps of I, II and III, as in the previous proofs of Theorem \ref{thm-n-5}, to determine the minimizer graphs in $\mathbb{G}^6_{n,\alpha}$.

{\flushleft\bf Step I.} To find $\mathcal{T}^*_-(6)$, we need consider  (\ref{eq-t1}) in case of $k=6$, which leads four possibilities:
$$
\footnotesize
\left\{\begin{array}{ll}d\in \{4,6,8,10\}\\
h\le \min\{ 5-\frac{d}{2}, \lfloor\frac{ d}{4}\rfloor \}\\
\sum^h_{s=1} |M_s|=5-\frac{d}{2}\\
\sigma^*:M_s\longrightarrow M_{s-1}\\
\ \ \mbox{ for }  1\le s \le h
\end{array}\right.\!\!\Rightarrow
\left\{\begin{array}{ll}d=4\\
h= 1\\
\sum^h_{s=1} |M_s|=3\\
\sigma^*:M_s\longrightarrow M_{s-1}\\
\ \ \mbox{ for }  1\le s \le h
\end{array}\right.\!\!\!\!,\
\left\{\begin{array}{ll}d=6\\
h= 1\\
\sum^h_{s=1} |M_s|=2\\
\sigma^*:M_s\longrightarrow M_{s-1}\\
\ \ \mbox{ for }  1\le s \le h
\end{array}\right.\!\!\!\!,\
\left\{\begin{array}{ll}d=8\\
h= 1\\
\sum^h_{s=1} |M_s|=1\\
\sigma^*:M_s\longrightarrow M_{s-1}\\
\ \ \mbox{ for }  1\le s \le h
\end{array}\right.
\!\!\!\!\mbox{or}
\left\{\begin{array}{ll}d=10\\
h=0\\
M_s=\emptyset
\end{array}\right.
$$
The first leads to $F^6_1=T(4; \{x_{11}y_{11}, x_{12}y_{12},x_{13}y_{13}\}; \sigma^*)$ (see Fig.\ref{fig-main-6});
the second leads two main trees $F^6_2=T(6; \{x_{11}y_{11}, x_{12}y_{12}\}; \sigma_1^*)$ and
 $F^6_3=T(6; \{x_{11}y_{11}, x_{12}y_{12}\}; \sigma_2^*)$ (see Fig.\ref{fig-main-6});
the third   leads  $F^6_4=T(6; \{x_{11}y_{11}\}; \sigma_3^*)$ and
 $F^6_5=T(6; \{x_{11}y_{11}\}; \sigma_4^*)$ (see Fig.\ref{fig-main-6});
 the later leads to $F^6_6=T(10; \emptyset; \sigma^*)=P_{11}$ (see Fig.\ref{fig-main-6}). Thus
 $\mathcal{T}^*_-(6)=\{F^6_1,F^6_2,F^6_3,F^6_4,F^6_5,F^6_6\}$.

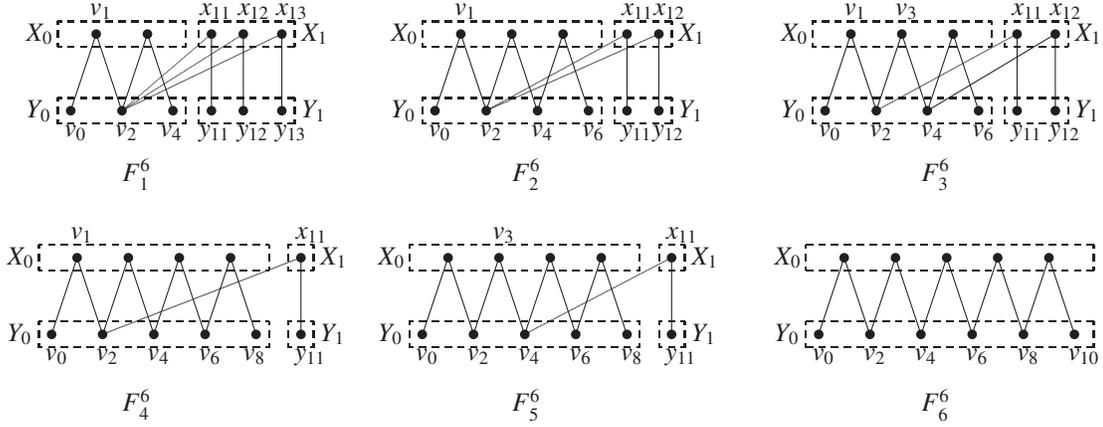
\begin{figure}
  \centering
  \footnotesize
% This is a LaTeX picture output by TeXCAD.
% File name: [Clipboard].
% Version of TeXCAD: 4.3
% Reference / build: 30-Jun-2012 (rev. 105)
% For new versions, check: http://texcad.sf.net/
% Options on the following lines.
%\grade{\on}
%\emlines{\off}
%\epic{\off}
%\beziermacro{\on}
%\reduce{\on}
%\snapping{\on}
%\pvinsert{% Your \input, \def, etc. here}
%\quality{8.000}
%\graddiff{0.005}
%\snapasp{1}
%\zoom{8.0000}
\unitlength 0.85mm % = 2.845pt
\linethickness{0.4pt}
\ifx\plotpoint\undefined\newsavebox{\plotpoint}\fi % GNUPLOT compatibility
\begin{picture}(169,62)(0,0)
\put(13,59){\circle*{1.5}}
\put(13,59){\line(1,-3){4}}
\put(17,47){\circle*{1.5}}
\put(9,47){\circle*{1.5}}
\put(13,59){\line(-1,-3){4}}
\put(21,59){\circle*{1.5}}
\put(21,59){\line(1,-3){4}}
\put(25,47){\circle*{1.5}}
\put(21,59){\line(-1,-3){4}}
\put(31,59){\circle*{1.5}}
\put(31,47){\circle*{1.5}}
\put(31,59){\line(0,-1){12}}
%\emline(31,59)(17,47)
\multiput(31,59)(-.0393258427,-.0337078652){356}{\line(-1,0){.0393258427}}
%\end
%\emline(36,59)(17,47)
\multiput(36,59)(-.0533707865,-.0337078652){356}{\line(-1,0){.0533707865}}
%\end
\put(2,58){$X_0$}
\put(2,46){$Y_0$}
\put(45,58){$X_1$}
\put(45,46){$Y_1$}
\put(36,59){\circle*{1.5}}
\put(36,47){\circle*{1.5}}
\put(36,59){\line(0,-1){12}}
\put(42,59){\circle*{1.5}}
\put(42,47){\circle*{1.5}}
\put(42,59){\line(0,-1){12}}
%\emline(42,59)(17,47)
\multiput(42,59)(-.0702247191,-.0337078652){356}{\line(-1,0){.0702247191}}
%\end
\put(10,24){\circle*{1.5}}
\put(10,24){\line(1,-3){4}}
\put(14,12){\circle*{1.5}}
\put(6,12){\circle*{1.5}}
\put(10,24){\line(-1,-3){4}}
\put(18,24){\circle*{1.5}}
\put(18,24){\line(1,-3){4}}
\put(22,12){\circle*{1.5}}
\put(18,24){\line(-1,-3){4}}
\put(26,24){\circle*{1.5}}
\put(26,24){\line(1,-3){4}}
\put(30,12){\circle*{1.5}}
\put(26,24){\line(-1,-3){4}}
\put(34,24){\circle*{1.5}}
\put(34,24){\line(1,-3){4}}
\put(38,12){\circle*{1.5}}
\put(34,24){\line(-1,-3){4}}
\put(-1,23){$X_0$}
\put(-1,11){$Y_0$}
\put(48,23){$X_1$}
\put(48,11){$Y_1$}
\put(45,24){\circle*{1.5}}
\put(45,12){\circle*{1.5}}
\put(45,24){\line(0,-1){12}}
%\emline(45,24)(14,12)
\multiput(45,24)(-.0870786517,-.0337078652){356}{\line(-1,0){.0870786517}}
%\end
\put(68,24){\circle*{1.5}}
\put(68,24){\line(1,-3){4}}
\put(72,12){\circle*{1.5}}
\put(64,12){\circle*{1.5}}
\put(68,24){\line(-1,-3){4}}
\put(76,24){\circle*{1.5}}
\put(76,24){\line(1,-3){4}}
\put(80,12){\circle*{1.5}}
\put(76,24){\line(-1,-3){4}}
\put(84,24){\circle*{1.5}}
\put(84,24){\line(1,-3){4}}
\put(88,12){\circle*{1.5}}
\put(84,24){\line(-1,-3){4}}
\put(92,24){\circle*{1.5}}
\put(92,24){\line(1,-3){4}}
\put(96,12){\circle*{1.5}}
\put(92,24){\line(-1,-3){4}}
\put(57,23){$X_0$}
\put(57,11){$Y_0$}
\put(106,23){$X_1$}
\put(106,11){$Y_1$}
\put(103,24){\circle*{1.5}}
\put(103,12){\circle*{1.5}}
\put(103,24){\line(0,-1){12}}
%\emline(103,24)(80,12)
\multiput(103,24)(-.0646067416,-.0337078652){356}{\line(-1,0){.0646067416}}
%\end
\put(130,24){\circle*{1.5}}
\put(130,24){\line(1,-3){4}}
\put(134,12){\circle*{1.5}}
\put(126,12){\circle*{1.5}}
\put(130,24){\line(-1,-3){4}}
\put(138,24){\circle*{1.5}}
\put(138,24){\line(1,-3){4}}
\put(142,12){\circle*{1.5}}
\put(138,24){\line(-1,-3){4}}
\put(146,24){\circle*{1.5}}
\put(146,24){\line(1,-3){4}}
\put(150,12){\circle*{1.5}}
\put(146,24){\line(-1,-3){4}}
\put(154,24){\circle*{1.5}}
\put(154,24){\line(1,-3){4}}
\put(158,12){\circle*{1.5}}
\put(154,24){\line(-1,-3){4}}
\put(119,23){$X_0$}
\put(119,11){$Y_0$}
\put(162,24){\circle*{1.5}}
\put(162,24){\line(1,-3){4}}
\put(166,12){\circle*{1.5}}
\put(162,24){\line(-1,-3){4}}
\put(29,62){$x_{11}$}
\put(35,62){$x_{12}$}
\put(41,62){$x_{13}$}
\put(29,43){$y_{11}$}
\put(35,43){$y_{12}$}
\put(41,43){$y_{13}$}
\put(8,43){$v_{0}$}
\put(16,43){$v_{2}$}
\put(23,43){$v_{4}$}
\put(12,62){$v_{1}$}
\put(17,36){$F^6_1$}
\put(44,27){$x_{11}$}
\put(44,8){$y_{11}$}
\put(5,8){$v_{0}$}
\put(13,8){$v_{2}$}
\put(21,8){$v_{4}$}
\put(29,8){$v_{6}$}
\put(36,8){$v_{8}$}
\put(9,27){$v_{1}$}
\put(17,0){$F^6_4$}
\put(102,27){$x_{11}$}
\put(102,8){$y_{11}$}
\put(63,8){$v_{0}$}
\put(71,8){$v_{2}$}
\put(79,8){$v_{4}$}
\put(87,8){$v_{6}$}
\put(95,8){$v_{8}$}
\put(75,27){$v_{3}$}
\put(78,0){$F^6_5$}
\put(125,8){$v_{0}$}
\put(133,8){$v_{2}$}
\put(141,8){$v_{4}$}
\put(149,8){$v_{6}$}
\put(157,8){$v_{8}$}
\put(165,8){$v_{10}$}
\put(142,0){$F^6_6$}
\put(7,45){\dashbox{1}(20,4)[cc]{}}
\put(7,57){\dashbox{1}(20,4)[cc]{}}
\put(29,45){\dashbox{1}(15,4)[cc]{}}
\put(29,57){\dashbox{1}(15,4)[cc]{}}
\put(4,10){\dashbox{1}(36,4)[cc]{}}
\put(4,22){\dashbox{1}(36,4)[cc]{}}
\put(43,10){\dashbox{1}(4,4)[cc]{}}
\put(43,22){\dashbox{1}(4,4)[cc]{}}
\put(62,10){\dashbox{1}(36,4)[cc]{}}
\put(62,22){\dashbox{1}(36,4)[cc]{}}
\put(101,22){\dashbox{1}(4,4)[cc]{}}
\put(101,10){\dashbox{1}(4,4)[cc]{}}
\put(124,10){\dashbox{1}(45,4)[cc]{}}
\put(124,22){\dashbox{1}(45,4)[cc]{}}
\put(70,59){\circle*{1.5}}
\put(70,59){\line(1,-3){4}}
\put(74,47){\circle*{1.5}}
\put(66,47){\circle*{1.5}}
\put(70,59){\line(-1,-3){4}}
\put(78,59){\circle*{1.5}}
\put(78,59){\line(1,-3){4}}
\put(82,47){\circle*{1.5}}
\put(78,59){\line(-1,-3){4}}
\put(86,59){\circle*{1.5}}
\put(86,59){\line(1,-3){4}}
\put(90,47){\circle*{1.5}}
\put(86,59){\line(-1,-3){4}}
\put(96,59){\circle*{1.5}}
\put(96,47){\circle*{1.5}}
\put(96,59){\line(0,-1){12}}
%\emline(96,59)(74,47)
\multiput(96,59)(-.0617977528,-.0337078652){356}{\line(-1,0){.0617977528}}
%\end
\put(59,58){$X_0$}
\put(59,46){$Y_0$}
\put(104,46){$Y_1$}
\put(104,58){$X_1$}
\put(101,59){\circle*{1.5}}
\put(101,47){\circle*{1.5}}
\put(101,59){\line(0,-1){12}}
%\emline(101,59)(74,47)
\multiput(101,59)(-.0758426966,-.0337078652){356}{\line(-1,0){.0758426966}}
%\end
\put(95,62){$x_{11}$}
\put(100,62){$x_{12}$}
\put(95,43){$y_{11}$}
\put(100,43){$y_{12}$}
\put(65,43){$v_{0}$}
\put(73,43){$v_{2}$}
\put(81,43){$v_{4}$}
\put(89,43){$v_{6}$}
\put(69,62){$v_{1}$}
\put(78,36){$F^6_2$}
\put(64,45){\dashbox{1}(28,4)[cc]{}}
\put(64,57){\dashbox{1}(28,4)[cc]{}}
\put(94,57){\dashbox{1}(9,4)[cc]{}}
\put(94,45){\dashbox{1}(9,4)[cc]{}}
\put(131,59){\circle*{1.5}}
\put(131,59){\line(1,-3){4}}
\put(135,47){\circle*{1.5}}
\put(127,47){\circle*{1.5}}
\put(131,59){\line(-1,-3){4}}
\put(139,59){\circle*{1.5}}
\put(139,59){\line(1,-3){4}}
\put(143,47){\circle*{1.5}}
\put(139,59){\line(-1,-3){4}}
\put(147,59){\circle*{1.5}}
\put(147,59){\line(1,-3){4}}
\put(151,47){\circle*{1.5}}
\put(147,59){\line(-1,-3){4}}
\put(157,59){\circle*{1.5}}
\put(157,47){\circle*{1.5}}
\put(157,59){\line(0,-1){12}}
%\emline(157,59)(135,47)
\multiput(157,59)(-.0617977528,-.0337078652){356}{\line(-1,0){.0617977528}}
%\end
\put(120,58){$X_0$}
\put(120,46){$Y_0$}
\put(166,46){$Y_1$}
\put(166,58){$X_1$}
\put(163,59){\circle*{1.5}}
\put(163,47){\circle*{1.5}}
\put(163,59){\line(0,-1){12}}
\put(163,59){\line(-5,-3){20}}
\put(156,62){$x_{11}$}
\put(162,62){$x_{12}$}
\put(156,43){$y_{11}$}
\put(162,43){$y_{12}$}
\put(126,43){$v_{0}$}
\put(134,43){$v_{2}$}
\put(142,43){$v_{4}$}
\put(150,43){$v_{6}$}
\put(130,62){$v_{1}$}
\put(138,62){$v_{3}$}
\put(142,36){$F^6_3$}
\put(125,45){\dashbox{1}(28,4)[cc]{}}
\put(125,57){\dashbox{1}(28,4)[cc]{}}
\put(155,57){\dashbox{1}(10,4)[cc]{}}
\put(155,45){\dashbox{1}(10,4)[cc]{}}
\end{picture}

  \caption{\footnotesize{All  main trees for $k=6$}}\label{fig-main-6}
\end{figure}

{\flushleft\bf Step II. } Let $n+1\equiv r$ ( $\mbox{mod }6$ ) and $n\ge n_0=3k^2-k-1-(k-1)r=101-5r$, where $0\le r\le 5$. To  determine the kernel $T_{n_0}$ of minimizer graph $T^*\in \mathbb{G}^6_{n,\alpha}$,
we need to traverse the main tree   $T_-^*=(V_1^*,V_2^*)\in \mathcal{T}^*_-(6)$ to determine the kernel  $T_{n_0}$.
 To exactly
 $T_{n_0}=T^*_{-}\circ l_{V_2^*}$ and its leaf sequence $l_{V_2^*}=(l(u)\mid u\in V_2^*)$ satisfying the condition
\begin{equation}\label{eq-con-6}
\left\{\begin{array}{ll}|L(T_{n_0})|=\sum_{u\in V_2^*}l(u)=n_0-11=90-5r,&\vspace{0.1cm}\\
 \bar{l}+r-5-d_{T^*_-}(u)\le l(u)\le \bar{l}+3-d_{T^*_{-}}(u)& \mbox{ if $0\le r \le 4$,}\vspace{0.1cm}\\
 \bar{l}-5-d_{T^*_-}(u)\le l(u)\le \bar{l}+4-d_{T^*_{-}}(u)&\mbox{ if $r=5$,}
  \end{array}\right.\end{equation}
where $\bar{l}=\lfloor\frac{|L(T_{n_0})|}{|V_2^*|}\rfloor=\lfloor\frac{90-5r}{6}\rfloor$.
As the same process as $k=5$, for each $0\le r\le 5$
we first obtain the set $\mathcal{L}_{F^6_i}$, whose elements are   $l_{V_2^*}$ satisfying (\ref{eq-con-6}) for the main tree $T^*_-=F^6_i\in \mathcal{T}^*_-(6)$
( \# of Tab.\ref{table-1-6} indicates the number of elements in  $\mathcal{L}_{F^6_i}$).  Let $\mathbb{F}^6_i=\{F^6_i\circ l_{V_2^*}\mid l_{V_2^*}\in \mathcal{L}_{F^6_i}\}$ for $i=1,...,6$.
Next, by comparing the spectral radii of  graphs in $\mathbb{F}^6_i$, we can get $T^6_i$ (see the seventh column in  Tab.\ref{table-1-6}) with the minimum spectral radius among $\mathbb{F}^6_i$. Finally, by comparing the spectral radii of $T^6_i$ ($i=1,...,6$), we obtain the kernel $T_{n_0}$ shown in the seventh column of Tab.\ref{table-1-6}.
{\flushleft\bf Step III. }
Let $n\ge n_0=101-5r$, where $0\le r\le 5$ and $\ell_{n,6}=\frac{n-n_0}{6}=\frac{101-5r}{6}$.
By Theorem \ref{thm-minimizer-spectral-radius} we have $T^*=T_{n_0}\circ \ell_{n,6} \mathbf{1}_{V_2^*}$ and $\rho(T^*)=\sqrt{\rho^2(T_{n_0})+\ell_{n,6}}$,
 where
  the kernel $T_{n_0}$ is the minimizer graph determined in Step II shown  in Tab.\ref{table-1-6} according to different $0\le r\le 5$.

Finally, we can state our result for $k=6$  in the following Theorem \ref{thm-n-6}.

\begin{table}[H]
\footnotesize
\caption{\small The  kernel  $T_{n_0}$ for $k=6$}
\centering
\renewcommand{\arraystretch}{1.2}
\begin{tabular*}{15.6cm}{p{5pt}|p{5pt}|p{8pt}|p{135pt}|p{10pt}|p{15pt}|p{130pt}|p{38pt}}
\hline
$k$&  $r$& $n_0$& the condition (\ref{eq-con-6})
  %satisfied by $(l(u)\mid u\in V_2^*)$
  & $T^*_-$ &  \#
  & graph $T^6_i$   & $\rho^2(T^6_i)$ \\ \hline
\multirow{44}*{6}
&\multirow{6}*{0}&\multirow{6}*{$101$}&
 \multirow{6}*{$\left\{\begin{array}{ll}\sum_{u\in V_2^*}l(u)=90\vspace{0.08cm}\\10-d_{T^*_-}(u)\le l(u)\le 18-d_{T^*_{-}}(u)
 \end{array}\right.$}&
 $F^6_1$&
$60$
 &
 $F^6_1\circ(16,10,16,16,16,16)$&
$16+\sqrt{6}$
 \\ \cline{5-8}
&&&&$F^6_2$&
$165$
&$F^6_2\circ(16,11,15,16,16,16)$&
$18.4370$
\\ \cline{5-8}
&&&&$F^6_3$&
$243$
&\bm{$F^6_3\circ(16,13,13,16,16,16)=T_{n_0}$}&
\bm{$17+\sqrt{2}$}
 \\
\cline{5-8}
&&&&$F^6_4$&
$791$
&\bm{$F^6_4\circ(16,13,14,15,16,16)=T_{n_0}$}&
\bm{$17+\sqrt{2}$}
 \\
\cline{5-8}
&&&&$F^6_5$&
$495$
&$F^6_5\circ(16,15,12,15,16,16)$&
$18.4309$
 \\
\cline{5-8}
&&&&$F^6_6$&
$651$
&\bm{$F^6_6\circ(16,15,14,14,15,16)=T_{n_0}$}&
\bm{$17+\sqrt{2}$}
 \\
\cline{2-8}
&\multirow{6}*{1}&\multirow{6}*{$96$}&
 \multirow{6}*{$\left\{\begin{array}{ll}\sum_{u\in V_2^*}l(u)=85\vspace{0.08cm}\\10-d_{T^*_-}(u)\le l(u)\le 17-d_{T^*_{-}}(u)
 \end{array}\right.$}&
 $F^6_1$&
$42$
 &
 $F^6_1\circ(16,5,16,16,16,16)$&
$\frac{27+\sqrt{69}}{2}$
\\ \cline{5-8}
&&&&$F^6_2$&
$120$
&$F^6_2\circ(16,8,13,16,16,16)$&
$17.6378$
\\ \cline{5-8}
&&&&$F^6_3$&
$154$
&$F^6_3\circ(16,10,11,16,16,16)$&
$17.6579$
 \\
\cline{5-8}
&&&&$F^6_4$&
$496$
&$F^6_4\circ(16,10,14,13,16,16)$&
$17.6323$
 \\
\cline{5-8}
&&&&\multirow{6}*{$F^6_5$}&
\multirow{6}*{$330$}
&\bm{$F^6_5\circ(15,14,11,13,16,16)=T_{n_0}$}&\multirow{6}*{\bm{$\frac{33+\sqrt{5}}{2}$}}\\
&&&&&&\bm{$F^6_5\circ(15,14,11,14,15,16)=T_{n_0}$}&\\
&&&&&&\bm{$F^6_5\circ(16,13,11,13,16,16)=T_{n_0}$}&\\
&&&&&&\bm{$F^6_5\circ(15,14,12,13,16,15)=T_{n_0}$}&\\
&&&&&&\bm{$F^6_5\circ(15,14,12,14,15,15)=T_{n_0}$}&\\
&&&&&&\bm{$F^6_5\circ(16,13,12,13,16,15)=T_{n_0}$}&\\
\cline{5-8}
&&&&\multirow{4}*{$F^6_6$}&
\multirow{4}*{$396$}
&\bm{$F^6_6\circ(15,14,13,14,14,15)=T_{n_0}$}&\multirow{4}*{\bm{$\frac{33+\sqrt{5}}{2}$}}\\
&&&&&&\bm{$F^6_6\circ(15,14,13,14,13,16)=T_{n_0}$}&\\
&&&&&&\bm{$F^6_6\circ(15,14,14,13,13,16)=T_{n_0}$}&\\
&&&&&&\bm{$F^6_6\circ(16,13,13,14,13,16)=T_{n_0}$}&\\
\cline{2-8}
&\multirow{6}*{2}&\multirow{6}*{$91$}&
 \multirow{6}*{$\left\{\begin{array}{ll}\sum_{u\in V_2^*}l(u)=80\vspace{0.08cm}\\10-d_{T^*_-}(u)\le l(u)\le 16-d_{T^*_{-}}(u)
 \end{array}\right.$}&
 $F^6_1$&
$29$
 &
 $F^6_1\circ(15,5,15,15,15,15)$&
$13+\sqrt{14}$
 \\ \cline{5-8}
&&&&$F^6_2$&
$84$
&$F^6_2\circ(15,8,12,15,15,15)$&
$16.7443$
\\ \cline{5-8}
&&&&$F^6_3$&
$95$
&\bm{$F^6_3\circ(15,10,10,15,15,15)=T_{n_0}$}&
\bm{$15+\sqrt{3}$}
 \\
\cline{5-8}
&&&&$F^6_4$&
$296$
&\bm{$F^6_4\circ(15,10,13,12,15,15)=T_{n_0}$}&
\bm{$15+\sqrt{3}$}
 \\
\cline{5-8}
&&&&$F^6_5$&
$210$
&$F^6_5\circ(15,12,11,12,15,15)$&
$16.7491$
 \\
\cline{5-8}
&&&&$F^6_6$&
$236$
&\bm{$F^6_6\circ(15,12,13,13,12,15)=T_{n_0}$}&
\bm{$15+\sqrt{3}$}
 \\
\cline{2-8}
&\multirow{6}*{3}&\multirow{6}*{$86$}&
 \multirow{6}*{$\left\{\begin{array}{ll}\sum_{u\in V_2^*}l(u)=75\vspace{0.08cm}\\10-d_{T^*_-}(u)\le l(u)\le 15-d_{T^*_{-}}(u)
 \end{array}\right.$}&
 $F^6_1$&
$19$
 &
 \bm{$F^6_1\circ(14,5,14,14,14,14)=T_{n_0}$}&
\bm{$\frac{25+\sqrt{45}}{2}$}
\\ \cline{5-8}
&&&&$F^6_2$&
$56$
&$F^6_2\circ(14,7,12,14,14,14)$&
$15.8664$
\\ \cline{5-8}
&&&&$F^6_3$&
$54$
&$F^6_3\circ(14,9,10,14,14,14)$&
$15.8830$
 \\
\cline{5-8}
&&&&$F^6_4$&
$166$
&$F^6_4\circ(14,10,11,12,14,14)$&
$15.8878$
 \\
\cline{5-8}
&&&&$F^6_5$&
$126$
&$F^6_5\circ(14,12,9,12,14,14)$&
$15.8750$
 \\
\cline{5-8}
&&&&$F^6_6$&
$126$
&$F^6_6\circ(14,12,11,12,12,14)$&
$15.8969$
 \\
\cline{2-8}
&\multirow{6}*{4}&\multirow{6}*{$81$}&
 \multirow{6}*{$\left\{\begin{array}{ll}\sum_{u\in V_2^*}l(u)=70\vspace{0.08cm}\\10-d_{T^*_-}(u)\le l(u)\le 14-d_{T^*_{-}}(u)
 \end{array}\right.$}&
 $F^6_1$&
$12$
 &
 \bm{$F^6_1\circ(13,5,13,13,13,13)=T_{n_0}$}&
\bm{$15$}
\\ \cline{5-8}
&&&&$F^6_2$&
$35$
&\bm{$F^6_2\circ(13,7,11,13,13,13)=T_{n_0}$}&
\bm{$15$}
\\ \cline{5-8}
&&&&$F^6_3$&
$30$
&\bm{$F^6_3\circ(13,9,9,13,13,13)=T_{n_0}$}&
\bm{$15$}
 \\
\cline{5-8}

&&&&$F^6_4$&
$86$
&\bm{$F^6_4\circ(13,9,11,11,13,13)=T_{n_0}$}&
\bm{$15$}
\\
\cline{5-8}
&&&&$F^6_5$&
$70$
&\bm{$F^6_5\circ(13,11,9,11,13,13)=T_{n_0}$}&
\bm{$15$}
 \\
\cline{5-8}
&&&&$F^6_6$&
$66$
&\bm{$F^6_6\circ(13,11,11,11,11,13)=T_{n_0}$}&
\bm{$15$}
 \\
\cline{2-8}
&\multirow{6}*{5}&\multirow{6}*{$76$}&
 \multirow{6}*{$\left\{\begin{array}{ll}\sum_{u\in V_2^*}l(u)=65\vspace{0.08cm}\\5-d_{T^*_-}(u)\le l(u)\le 14-d_{T^*_{-}}(u)
 \end{array}\right.$}&
 $F^6_1$&
$83$
 &
 \bm{$F^6_1\circ(12,5,12,12,12,12)=T_{n_0}$}&
\bm{$\frac{23+\sqrt{29}}{2}$}
\\ \cline{5-8}
&&&&$F^6_2$&
$220$
&$F^6_2\circ(12,7,10,12,12,12)$&
$14.2080$
\\ \cline{5-8}
&&&&$F^6_3$&
$364$
&$F^6_3\circ(12,8,9,12,12,12)$&
$14.2361$
 \\
\cline{5-8}

&&&&$F^6_4$&
$1211$
&$F^6_4\circ(12,9,10,10,12,12)$&
$14.2470$
 \\
\cline{5-8}
&&&&$F^6_5$&
$715$
&$F^6_5\circ(12,10,9,10,12,12)$&
$14.2283$
 \\
\cline{5-8}
&&&&$F^6_6$&
$1001$
&$F^6_6\circ(12,10,10,11,10,12)$&
$14.2831$
\\
\hline
\end{tabular*}\label{table-1-6}
\end{table}

\begin{thm}\label{thm-n-6}
Let $T^*$ be a minimizer  graph in $\mathbb{G}^6_{n,\alpha}$ and let
$n+1\equiv r$ ( $\mbox{mod }6$ ), where $0\le r \le 5$. For $n\ge 101-5r$,  all  the  minimizer  graphs in $\mathbb{G}^6_{n,\alpha}$ and their spectral radii are listed in Tab.\ref{table-k-6},
in which $F^6_i$ are shown in Fig.\ref{fig-main-6} for $i=1,2,...,6$.
\end{thm}
\begin{table}[H]
\footnotesize
\caption{\small The  minimizer  graph  $T^*$ and its spectral radius ( $n+1\equiv r$ ( $\mbox{mod }6$ ) )}
\centering
\renewcommand{\arraystretch}{1.35}
\begin{tabular*}{15.2cm}{p{2pt}|p{144pt}|p{36pt}||p{2pt}|p{136pt}|p{36pt}}
\hline
  $r$& $T^*$& $\rho(T^*)$&$r$& $T^*$& $\rho(T^*)$\\ \hline
\multirow{3}*{$0$}&
$F^6_3\!\circ\!(\frac{n-5}{6},\frac{n-23}{6},\frac{n-23}{6},\frac{n-5}{6},\frac{n-5}{6},\frac{n-5}{6})$
&\multirow{3}*{\!\!\!$\sqrt{\frac{n+1}{6}\!+\!\!\sqrt{2}}$}&
\multirow{3}*{$2$}&
$F^6_3\!\circ\!(\frac{n-1}{6},\frac{n-31}{6},\frac{n-31}{6},\frac{n-1}{6},\frac{n-1}{6},\frac{n-1}{6})$
&\multirow{3}*{\!\!\!$\sqrt{\frac{n-1}{6}\!+\!\!\sqrt{3}}$}
\\
&$F^6_4\!\circ\!(\frac{n-5}{6},\frac{n-23}{6},\frac{n-17}{6},\frac{n-11}{6},\frac{n-5}{6},\frac{n-5}{6})$&
&&$F^6_4\!\circ\!(\frac{n-1}{6},\frac{n-31}{6},\frac{n-13}{6},\frac{n-19}{6},\frac{n-1}{6},\frac{n-1}{6})$&\\
&$F^6_6\!\circ\!(\frac{n-5}{6},\frac{n-11}{6},\frac{n-17}{6},\frac{n-17}{6},\frac{n-11}{6},\frac{n-5}{6})$&
&&$F^6_6\!\circ\!(\frac{n-1}{6},\!\frac{n-19}{6},\frac{n-13}{6},\frac{n-13}{6},\frac{n-19}{6},\frac{n-1}{6})$&\\
\cline{1-6}
\multirow{10}*{$1$}&
$F^6_5\!\circ\!(\frac{n}{6}\!-\!1,\frac{n}{6}-2,\frac{n}{6}-5,\frac{n}{6}-3,\frac{n}{6},\frac{n}{6})$
&\multirow{10}*{\!\!\!$\sqrt{\frac{n}{6}+\!\frac{1+\sqrt{5}}{2}}$}&
\multirow{2}*{$3$}&
\multirow{2}*{$F^6_1\!\circ\!(\frac{n-2}{6},\frac{n-56}{6},\frac{n-2}{6},\frac{n-2}{6},\frac{n-2}{6},\frac{n-2}{6})$}
&\multirow{2}*{$\!\!\!\sqrt{\frac{n\!-11}{6}\!+\!\!\frac{\sqrt{45}}{2}}$}\\

&$F^6_5\!\circ\!(\frac{n}{6}\!-\!1,\frac{n}{6}\!-\!2,\frac{n}{6}\!-\!5,\frac{n}{6}\!-\!2,\frac{n}{6}\!-\!1,\frac{n}{6})$&
&&&\\\cline{4-6}

&$F^6_5\!\circ\!(\frac{n}{6},\frac{n}{6}-3,\frac{n}{6}-5,\frac{n}{6}-3,\frac{n}{6},\frac{n}{6})$&
&\multirow{6}*{$4$}&
$F^6_1\!\circ\!(\frac{n-3}{6},\frac{n-51}{6},\frac{n-3}{6},\frac{n-3}{6},\frac{n-3}{6},\frac{n-3}{6})$
&\multirow{6}*{$\sqrt{\frac{n+9}{6}}$}
\\
&$F^6_5\!\circ\!(\frac{n}{6}\!-\!1,\frac{n}{6}\!-\!2,\frac{n}{6}\!-\!4,\frac{n}{6}\!-\!3,\frac{n}{6},\frac{n}{6}\!-\!1)$&
&
&$F^6_2\!\circ\!(\frac{n-3}{6},\frac{n-39}{6},\frac{n-15}{6},\frac{n-3}{6},\frac{n-3}{6},\frac{n-3}{6})$&\\

&$F^6_5\!\circ\!(\frac{n}{6}\!-\!1,\frac{n}{6}\!-\!2,\frac{n}{6}\!-\!4,\frac{n}{6}\!-\!2,\frac{n}{6}\!-\!1,\frac{n}{6}\!-\!1)$&
&&$F^6_3\!\circ\!(\frac{n-3}{6},\frac{n-27}{6},\frac{n-27}{6},\frac{n-3}{6},\frac{n-3}{6},\frac{n-3}{6})$&\\

&$F^6_5\!\circ\!(\frac{n}{6},\frac{n}{6}-3,\frac{n}{6}-4,\frac{n}{6}-3,\frac{n}{6},\frac{n}{6}-1)$&
&
&$F^6_4\!\circ\!(\frac{n-3}{6},\frac{n-27}{6},\frac{n-15}{6},\frac{n-15}{6},\frac{n-3}{6},\frac{n-3}{6})$&\\

&$F^6_6\!\circ\!(\frac{n}{6}\!-\!1,\frac{n}{6}\!-\!2,\frac{n}{6}\!-\!3,\frac{n}{6}\!-\!2,\frac{n}{6}\!-\!2,\frac{n}{6}\!-\!1)$&
&
&$F^6_5\!\circ\!(\frac{n-3}{6},\frac{n-15}{6},\frac{n-27}{6},\frac{n-15}{6},\frac{n-3}{6},\frac{n-3}{6})$&\\

&$F^6_6\!\circ\!(\frac{n}{6}\!-\!1,\frac{n}{6}\!-\!2,\frac{n}{6}\!-\!3,\frac{n}{6}\!-\!2,\frac{n}{6}\!-\!3,\frac{n}{6})$&
&&$F^6_6\!\circ\!(\frac{n-3}{6},\!\frac{n-15}{6},\frac{n-15}{6},\frac{n-15}{6},\frac{n-15}{6},\frac{n-3}{6})$&\\
\cline{4-6}

&$F^6_6\!\circ\!(\frac{n}{6}\!-\!1,\frac{n}{6}\!-\!2,\frac{n}{6}\!-\!2,\frac{n}{6}\!-\!3,\frac{n}{6}\!-\!3,\frac{n}{6})$&&\multirow{2}*{$5$}&
\multirow{2}*{$F^6_1\!\circ\!(\frac{n-4}{6},\frac{n-46}{6},\frac{n-4}{6},\frac{n-4}{6},\frac{n-4}{6},\frac{n-4}{6})$}
&\multirow{2}*{$\!\!\!\sqrt{\frac{n\!-7}{6}\!+\!\!\frac{\sqrt{29}}{2}}$}\\
&$F^6_6\!\circ\!(\frac{n}{6},\frac{n}{6}-3,\frac{n}{6}-3,\frac{n}{6}-2,\frac{n}{6}-3,\frac{n}{6})$&&&&\\
\hline
\end{tabular*}\label{table-k-6}
\end{table}

\begin{remark}
The key to determining a  minimizer graph $T^*$ is to determine its kernel, which is construct from a main trees in the set  $\mathcal{T}_{-}^*(k)$. The number of the main trees  increase as $k $ increases, and we see from known results, in particular our Theorem \ref{thm-n-5} and Theorem \ref{thm-n-6}, that every main tree in $\mathcal{T}_{-}^*(k)$ generates at least one minimizer graph.
\end{remark}

\section{Conclusion}
Theoretically, Theorem \ref{thm-minimizer-spectral-radius} together with Theorem \ref{thm-minimizer-main} completely characterize the minimizer graph and its spectral radius  in $\mathbb{G}_{n,\alpha}^k$ when $k=n-\alpha \le\frac{n}{2}$,
moreover we give a general method for determining the minimizer graphs in three specific steps.
However, the characterization of minimizer graph and its spectral radius for $k=n-\alpha >\frac{n}{2}$ is still an open problem.

It is clear from our results that, on the one hand, although the  minimizer graph $T^*=T_{n_0}\circ \ell_{n,k} \mathbf{1}_{V_2^*}$ and its spectral radius $\rho(T^*)=\sqrt{\rho^2(T_{n_0})+\ell_{n,k}}$ have uniform expressions depending on its kernel $T_{n_0}$, the  minimizer graphs are generally not unique and their tree structures are diverse as  increases with $k$.
 On the other hand, the representation of the  minimizer graph depends on the classification of $n$ by mod $k$.
The determination of the minimizer graph and the calculation of its spectral radius have a certain complexity.
Given $k$, the kernel $T_{n_0}$ is the minimal graph in $\mathbb{G}_{n_0,n_0-k}^k$, where $n_0=3k^2-k-1-(k-1)r$, and when $k$ is small (e.g., $k=1, 2,... ,6$) we can simply determine $T_{n_0}$ from the structural features of the minimizer graph obtained in this paper,
thus giving minimizer graphs of arbitrary order $n\ge n_0$. However, as $k$ increases $n_0$ grows by the square order of $k$, the kernel of the minimizer graph can only be found with the help of a computer.
As Stevanovi\'{c} pointed out in \cite{Stevanovic},  determining the  graph with the minimum spectral radius among connected graph with independence number  $\alpha$ appears  to be a tough problem.

\end{document}